\documentclass{amsart}

\newtheorem{theorem}{Theorem}[section]
\newtheorem{lemma}[theorem]{Lemma}
\newtheorem{prop}[theorem]{Proposition}
\newtheorem{coro}[theorem]{Corollary}

\newtheorem{problem}{Open Problem}[section]
\theoremstyle{definition}

\theoremstyle{remark}
\newtheorem{remark}[theorem]{Remark}

\numberwithin{equation}{section}

\begin{document}

\title[extension and convergence of mean curvature flow]{The extension and convergence of mean curvature flow in higher
codimension}

\author{Kefeng Liu}
\address{Center of Mathematical Sciences, Zhejiang University,
             Hangzhou,  310027, People¡¯s Republic of China; Department of Mathematics, UCLA, Box 951555, Los Angeles, CA, 90095-1555 }
\email{liu@cms.zju.edu.cn, liu@math.ucla.edu}

\author{Hongwei Xu}

\address{Center of Mathematical Sciences, Zhejiang University,
             Hangzhou,  310027, People¡¯s Republic of China}
\email{xuhw@cms.zju.edu.cn}

\author{Fei Ye}
\address{Center of Mathematical Sciences, Zhejiang University,
             Hangzhou,  310027, People¡¯s Republic of China}
\email{yf@cms.zju.edu.cn}

\author{Entao Zhao}
\address{Center of Mathematical Sciences, Zhejiang University,
             Hangzhou,  310027, People¡¯s Republic of China}
\email{zhaoet@cms.zju.edu.cn}

\thanks{Research supported by the National Natural Science Foundation
of China, Grant No. 11071211; the Trans-Century Training Programme
Foundation for Talents by the Ministry of Education of China, and
the China Postdoctoral Science Foundation, Grant No. 20090461379.}

\subjclass[2000]{53C44, 53C40}

\date{}

\keywords{Mean curvature flow, submanifold, maximal existence time,
convergence theorem, integral curvature}

\begin{abstract}In this paper, we first investigate the integral curvature condition to extend the mean
curvature flow of submanifolds in a Riemannian manifold with
codimension $d\geq1$, which generalizes the extension theorem for
the mean curvature flow of hypersurfaces  due to Le-\v{S}e\v{s}um
\cite{LS} and the authors \cite{XYZ1,XYZ2}. Using the extension
theorem, we prove two convergence theorems for the mean curvature
flow of closed submanifolds in ${R}^{n+d}$ under suitable integral
curvature conditions.
\end{abstract}

\maketitle

\section{Introduction}

\label{intro}Let $F_0:M^{n}\rightarrow N^{n+d}$ be a smooth
immersion from an $n$-dimensional Riemannian manifold without
boundary to an $(n+d)$-dimensional Riemannian manifold. Consider a
one-parameter family of smooth immersions $F:M\times
[0,T)\rightarrow N$ satisfying
\begin{eqnarray*}\left\{
\begin{array}{ll}
\left(\frac{\partial}{\partial t}F(x,t)\right)^{\bot}=H(x,t),\\
F(x,0)=F_0(x),
\end{array}\right.
\end{eqnarray*}
where $\left(\frac{\partial}{\partial t}F(x,t)\right)^{\bot}$ is the
normal component of $\frac{\partial}{\partial t}F(x,t)$, $H(x,t)$ is
the mean curvature vector of $F_t(M)$ and $F_t(x)=F(x,t)$. We call
$F:M\times [0,T)\rightarrow N$ the mean curvature flow with initial
value $F_0:M\rightarrow N$. This is the general form of the mean
curvature flow, which is  a nonlinear weakly parabolic system and is
invariant under reparametrization of $M$. We can find a family of
diffeomorphisms $\phi_t:M\rightarrow M$ for $t\in [0,T)$ such that
$\bar{F}_t=F_t\circ\phi_t: M\rightarrow N$ satisfies
$\frac{\partial}{\partial t}\bar{F}(x,t)=\bar{H}(x,t)$. We will
study the (reparameterized) mean curvature flow

\begin{eqnarray}
\label{MCF}\left\{
\begin{array}{ll}
\frac{\partial}{\partial t}F(x,t)=H(x,t),\\
F(x,0)=F_0(x).
\end{array}\right.
\end{eqnarray}

In \cite{B}, Brakke introduced the motion of a submanifold by its
mean curvature in arbitrary codimension and constructed a
generalized varifold solution for all time. For the classical
solution of the mean curvature flow, most works have been done on
hypersurfaces. Huisken \cite{H1,H2} showed that if the second
fundamental form is uniformly bounded, then the mean curvature flow
can be extended over the time. He then proved that if the initial
hypersurface in a complete manifold with bounded geometry is compact
and uniformly convex,  then the mean curvature flow converges to a
round point in finite time. Many other beautiful results have been
obtained, and there are various  approaches to study the mean
curvature flow of hypersurfaces (see \cite{CGG,ES}, etc.). However,
relatively little is known about the mean curvature flows of
submanifolds in higher codimensions, see \cite{Sm,SW,WaM1,WaM2,WaM3}
etc. for example. Recently, Andrews-Baker \cite{Andrews-Baker}
proved a convergence theorem for the mean curvature flow of closed
submanifolds satisfying suitable pinching condition in the Euclidean
space.

On the other hand,  Le-\v{S}e\v{s}um \cite{LS} and Xu-Ye-Zhao
\cite{XYZ1} obtained some integral conditions to extend the mean
curvature flow of hypersurfaces in the Euclidean space
independently. Later, Xu-Ye-Zhao \cite{XYZ2} generalized these
extension theorems to the case where the ambient space is a
Riemannian manifold with bounded geometry.

For an $n$-dimensional  submanifold $M$ in a Riemannian manifold, we
denote by $g$ the induced metric on $M$. Let $A$ and $H$ be the
second fundamental form and the mean curvature vector of $M$,
respectively. In this paper, we first generalize the extension
theorems in \cite{LS,XYZ1,XYZ2} to the mean curvature flow of
submanifolds in a Riemannian manifold with bounded geometry.

\begin{theorem}\label{main-extension}
Let $F_t:M^n\rightarrow N^{n+d}$ $(n\geq3)$ be the mean curvature
flow solution of closed submanifolds in a finite time interval $[0,T)$, where $N$ has bounded geometry. If\\
(i) there exist positive constants $a$ and $b$ such that $|A|^2\leq a|H|^2+b$ for $t\in [0,T)$,\\
(ii) $||H||_{\alpha,M\times[0,T)}=\left(\int_0^T\int_{M_t}|H|^\alpha
d\mu_tdt\right)^{\frac{1}{\alpha}}<\infty$ for some
$\alpha\geq n+2$,\\
then this flow can be extended over time $T$.
\end{theorem}

Let $\mathring{A}$ be the tracefree second fundamental form, which
is defined by $\mathring{A}=A-\frac{1}{n}g\otimes H$. Denote by
$||\cdot||_p$ the $L^p$-norm of a function or a tensor field. We
obtain the following convergence theorems for the mean curvature
flow of closed submanifolds in the Euclidean space.

\begin{theorem}\label{main-convergence-A}
Let $F:M^n\rightarrow R^{n+d}$ $(n\geq3)$ be a smooth closed
submanifold.  Then for any fixed $p>1$, there is a positive constant
$C_1$ depending on $n,p, Vol(M)$ and $||A||_{n+2}$, such that if
\begin{equation*}||\mathring{A}||_{p}<C_1,\end{equation*}
 then the mean curvature flow with $F$
as initial value has a unique solution $F:M\times[0,T)\rightarrow
R^{n+d}$ in a finite maximal time interval, and $F_t$ converges
uniformly to a point $x\in R^{n+d}$ as $t\rightarrow T$. The
rescaled maps $\widetilde{F}_t=\frac{F_t-x}{\sqrt{2n(T-t)}}$
converge in $C^{\infty}$ to a limiting embedding $\widetilde{F}_T$
such that $\widetilde{F}_T(M)$ is the unit $n$-sphere in some
$(n+1)$-dimensional subspace of $R^{n+d}$.
\end{theorem}

\begin{theorem}\label{main-convergence-H}
Let $F:M^n\rightarrow R^{n+d}$ $(n\geq3)$ be a smooth closed
submanifold. Then for any fixed $p>n$, there is a positive constant
$C_2$ depending on $n,p, Vol(M)$ and $||H||_{n+2}$, such that if
\begin{equation*}||\mathring{A}||_{p}<C_2,\end{equation*}
then the mean curvature flow with $F$ as initial value has a unique
solution $F:M\times[0,T)\rightarrow R^{n+d}$ in a finite maximal
time interval, and $F_t$ converges uniformly to a point $x\in
R^{n+d}$ as $t\rightarrow T$. The rescaled maps
$\widetilde{F}_t=\frac{F_t-x}{\sqrt{2n(T-t)}}$ converge in
$C^{\infty}$ to a limiting embedding $\widetilde{F}_T$  such that
$\widetilde{F}_T(M)$ is the unit $n$-sphere in some
$(n+1)$-dimensional subspace of $R^{n+d}$.\end{theorem}

As immediate consequences of the convergence theorems, we obtain the
following differentiable sphere theorems. First let $C_1$ be as in Theorem 1.2, we have
\begin{coro}
Let $F:M^n\rightarrow R^{n+d}$ $(n\geq3)$ be a smooth closed
submanifold.  If
\begin{equation*}||\mathring{A}||_{p}<C_1,\end{equation*}
for some $p>1$, then $M$ is diffeomorphic to the unit $n$-sphere.
\end{coro}
Similarly let $C_2$ be as Theorem 1.3, we have
\begin{coro}
Let $F:M^n\rightarrow R^{n+d}$ $(n\geq3)$ be a smooth closed
submanifold. If
\begin{equation*}||\mathring{A}||_{p}<C_2,\end{equation*}
for some $p>n$, then $M$ is diffeomorphic to the unit $n$-sphere.
\end{coro}

We remark that in the above theorems and corollaries, we can replace
the volume $Vol(M)$ by a positive lower bound of $|H|$ in which case our method works
without change.

The paper is organized as follows. In section 2, we introduce some
basic equations in submanifold theory, and recall the evolution
equations of the second fundamental form along the mean curvature
flow. In section 3, by using the Moser iteration and blow-up method
for parabolic equations, we prove Theorem \ref{main-extension}.
Theorems \ref{main-convergence-A} and \ref{main-convergence-H} are
proved in section 4. In section 5, we propose some unsolved problems
on convergence of the mean curvature flow in higher codimension.

\section{Preliminaries}

\label{Pre} Let $F:M^{n}\rightarrow N^{n+d}$ be a smooth immersion
from an $n$-dimensional Riemannian manifold $M^n$ without boundary
to an $(n+d)$-dimensional Riemannian manifold $N^{n+d}$. We shall
make use of the following convention on the range of indices.
$$1\leq i,j,k,\cdots \leq n,\ \ 1\leq A,B,C, \cdots \leq n+d,\ \
and\ \ n+1\leq\alpha,\beta,\gamma, \cdots \leq n+d.$$ The Einstein
sum convention is used to sum over the repeated indices.

 Suppose $\{x^i\}$ is a local coordinate system on $M$ and $\{y^A\}$ is a
local coordinate system on $N$. The metric $g=\sum g_{ij}dx^i\otimes
dx^j$ on $M$ induced from the metric $\langle\ ,\ \rangle$ on $N$ by
$F$ is
\begin{equation*}g_{ij}=\bigg\langle F_{\ast}\Big(\frac{\partial}{\partial x^i}\Big) ,F_{\ast}\Big(\frac{\partial}{\partial x^j}\Big)\bigg\rangle.
\end{equation*} The volume form on $M$ is
$d\mu=\sqrt{\det(g_{ij})}dx$.

 For any $x\in M$, denoted by $N_xM$
the normal space of $M$ in $N$ at point $x$, which is  the
orthogonal complement of $T_xM$ in $F^{\ast}T_{F(x)}N$. Denote by
$\bar{\nabla}$ the Levi-Civita connection on $N$. The Riemannian
curvature tensor  $\bar{R}$  of $N$ is defined by
\begin{equation*}\bar{R}(U,V)W=-\bar{\nabla}_U\bar{\nabla}_VW+\bar{\nabla}_V\bar{\nabla}_UW+\bar{\nabla}_{[U,V]}W
\end{equation*}
for vector fields $U,V$ and $ W$ tangent to $N$. The induced
connection $\nabla$ on $M$ is defined by
\begin{equation*}\nabla_XY=(\bar{\nabla}_XY)^{\top},
\end{equation*} for $X,Y$ tangent to $M$, where $(\ )^\top$ denotes
tangential component. Let $R$ be the Riemannian curvature tensor of
$M$.

Given a normal vector field $\xi$ along $M$, the induced connection
$\nabla^\bot$ on the normal bundle is defined by
\begin{equation*}\nabla^\bot
_X\xi=(\bar{\nabla}_X\xi)^{\bot},\end{equation*} where $(\ )^{\bot}$
denotes the normal component. Let $R^\bot$ denote the normal
curvature tensor.

 The second fundamental form is defined to be
$$A(X,Y)=(\bar{\nabla}_XY)^\bot$$ as a section of the tensor bundle
$T^\ast M\otimes T^\ast M\otimes NM$, where $T^\ast M$ and $NM$ are
the cotangential bundle and  the normal bundle over $M$. The mean
curvature vector $H$ is the trace of the second fundamental form.

The first covariant derivative of $A$ is defined as
\begin{equation*}
(\widetilde{\nabla}_XA)(Y,Z)=\nabla^\bot_XA(Y,Z)-A(\nabla_XY,Z)-A(Y,\nabla_XZ),
\end{equation*}
where $\widetilde{\nabla}$ is the connection on $T^\ast M\otimes
T^\ast M\otimes NM$. Similarly, we can define the second covariant
derivative of $A$.

Choosing orthonormal bases $\{e_i\}_{i=1}^n$ for $T_xM$ and
$\{e_\alpha\}_{\alpha=n+1}^{n+d}$ for $N_xM$, the components of the
second fundamental form and its first and second covariant
derivatives are
\begin{equation*}\begin{split}
h^\alpha_{ij}=&\langle
A(e_i,e_j), e_\alpha\rangle,\\
h^\alpha_{ijk}=&\langle (\widetilde{\nabla} _{e_k}A)(e_i,e_j),
e_\alpha\rangle,\\
h^\alpha_{ijkl}=&\langle (\widetilde{\nabla}
_{e_l}\widetilde{\nabla}_{e_k}A)(e_i,e_j), e_\alpha\rangle.
\end{split}
\end{equation*}
The Laplacian of $A$ is defined by $\Delta
h^\alpha_{ij}=\sum_kh^{\alpha}_{ijkk}$.

We define the tracefree second fundamental form $\mathring{A}$ by
$\mathring{A}=A-\frac{1}{n}g\otimes H$, whose components are
$\mathring{A}^\alpha_{ij}=h^\alpha_{ij}-\frac{1}{n}h^\alpha_{kk}\delta_{ij}$.
Obviously, we have $\mathring{A}^\alpha_{ii}=0$.

Let \begin{equation*}
\begin{split}
R_{ijkl}&=g(R(e_i,e_j)e_k,e_l),\\
\bar{R}_{ABCD}&=\langle\bar{R}(e_A,e_B)e_C,e_D\rangle,\\
R^\bot_{ij\alpha\beta}&=\langle{R^\bot}(e_i,e_j)e_\alpha,e_\beta\rangle.
\end{split}
\end{equation*}
Then we have the following Gauss, Codazzi and Ricci equations.
\begin{equation*}
\begin{split}
R_{ijkl}&=\bar{R}_{ijkl}+h^\alpha_{ik}h^\alpha_{jl} -h^\alpha_{il}h^\alpha_{jk},\\
h^\alpha_{ijk}-h^\alpha_{ikj}&=-\bar{R}_{\alpha ijk},\\
R^\bot_{ij\alpha\beta}&=\bar{R}_{ij\alpha\beta}+h^\alpha_{ik}h^\beta_{jk}-h^\alpha_{jk}h^\beta_{ik}.
\end{split}
\end{equation*}

Suppose $F:M\times [0,T)\rightarrow N$ is the mean curvature flow
with initial value $F_0:M\rightarrow N$. We have the following
evolution equations.
\begin{lemma}[\cite{WaM1}]\label{evo-equ} Along
the mean curvature flow we have
\begin{equation}\label{evo-d mu}
\frac{\partial}{\partial t}d\mu_t=-|H|^2d\mu_t, \end{equation}
\begin{equation}\label{evo-A}
\begin{split}
\frac{\partial}{\partial t}h^\alpha_{ij}=&\Delta
h^\alpha_{ij}+\bar{R}_{\alpha ijk,k}+\bar{R}_{\alpha
kik,j}\\
&-2\bar{R}_{lijk}h^\alpha_{lk}+2\bar{R}_{\alpha \beta
jk}h^\beta_{ik}+2\bar{R}_{\alpha\beta ik}h^\beta_{jk}\\
&-\bar{R}_{lkik}h^\alpha_{lj}-\bar{R}_{lkjk}h^\alpha_{li}+\bar{R}_{\alpha k\beta k}h^{\beta}_{ij}\\
&-h^\alpha_{im}(h^\beta_{jm}h^\beta_{ll}-h^\beta_{km}h^\beta_{jk})\\
&-h^\alpha_{km}(h^\beta_{jm}h^\beta_{ik}-h^\beta_{km}h^\beta_{ij})\\
&-h^\beta_{ik}(h^\beta_{jl}h^\alpha_{kl}-h^\beta_{kl}h^\alpha_{jl})\\
&-h^\alpha_{jk}h^\beta_{ij}h^\beta_{ll}+h^\beta \langle e_\alpha,
\bar{\nabla}_H e_\beta \rangle,\\
\end{split}
\end{equation}
where $\bar{R}_{ABCD,E}$ are the components of the first covariant
derivative $\bar{\nabla}\bar{R}$ of $\bar{R}$.
\end{lemma}

Throughout of the paper, we  assume that the submanifold is
connected,  and the ambient space $N$ has bounded geometry. Recall
that a Riemannian manifold is said to have bounded geometry if (i)
the sectional curvature is bounded; (ii) the injective radius is
bounded from below by a positive constant. We always assume that $N$
is a Riemannian manifold with bounded geometry satisfying $-K_1\leq
K_N\leq K_2$ for nonnegative constant $K_1$, $K_2$, and the
injective radius of $N$ is bounded from below by a positive constant
$i_N$.

\section{The extension of  mean curvature flow}

In this section, we prove the extension theorem for the mean
curvature flow of submanifolds in arbitrary codimension. The
following Sobolev inequality can be found in \cite{HS}.

\begin{lemma}[\cite{HS}]\label{Sobo-ineq-HS}
Let $M^n\subset N^{n+d}$ be an $n(\geq2)$-dimensional closed
submanifold in a Riemannian manifold $N^{n+d}$ with codimension
$d\geq1$. Denote by $i_N$ the positive lower bound of the injective
radius of $N$ restricted on $M$. Assume the sectional curvature
$K_N$ of $N$ satisfies $K_N\leq b^2$. Let $h$ be a non-negative
$C^1$ function on $M$. Then
$$\left(\int_Mh^{\frac{n}{n-1}}d\mu\right)^{\frac{n-1}{n}}\leq
C(n,\alpha)\int_M\Big[|\nabla h|+h|H|\Big]d\mu,$$ provided
$$b^2(1-\alpha)^{-\frac{2}{n}}(\omega_n^{-1}Vol(supp\ h))^{\frac{2}{n}}\leq1\ and \ 2\rho_0\leq i_N,$$ where
\[\rho_0=\left\{ \begin{array}{ll}
b^{-1}\sin^{-1}b(1-\alpha)^{-\frac{1}{n}}(\omega_n^{-1}Vol(supp\ h))^{\frac{1}{n}}\ \ \ &\ for\ b\ real,\\
(1-\alpha)^{-\frac{1}{n}}(\omega_n^{-1}Vol(supp\ h))^{\frac{1}{n}}\
\ \ &\ for\ b\ imaginary.
\end{array}
\right.\] Here $\alpha$ is a free parameter, $0<\alpha<1$, and
$$C(n,\alpha)=\frac{1}{2}\pi\cdot2^{n-2}\alpha^{-1}(1-\alpha)^{-\frac{1}{n}}\frac{n}{n-1}\omega_n^{-\frac{1}{n}}.$$
For $b$ imaginary, we may omit the factor $\frac{1}{2}\pi$ in the
definition of $C(n,\alpha)$.\end{lemma}

\begin{lemma}\label{Sobo-ineq}
Let $M^n\subset N^{n+d}$ be an $n(\geq3)$-dimensional closed
submanifold in a Riemannian manifold $N^{n+d}$ with codimension
$d\geq1$.  Assume $K_N\leq K_2$, where $K_2$ is a nonnegative
constant. Let $f$ be a non-negative $C^1$ function on $M$ satisfying
\begin{equation}
\label{cond-1}K_2(n+1)^{\frac{2}{n}}(\omega_n^{-1}Vol(supp\
f))^{\frac{2}{n}}\leq1,\end{equation}
\begin{equation}\label{cond-2}2K_2^{-\frac{1}{2}}\sin^{-1}K_2^{\frac{1}{2}}(n+1)^{\frac{1}{n}}(\omega_n^{-1}Vol(supp\
f))^{\frac{1}{n}}\leq i_N.
\end{equation} Then
$$||\nabla f||^2_2\geq\frac{(n-2)^2}{4(n-1)^2(1+s)}\left[\frac{1}{C^2(n)}
|| f||^2_{\frac{2n}{n-2}}-H^2_0\left(1+\frac{1}{s}\right)||
f||^2_2\right],$$ where $H_0=\max_{x\in M}|H|$,
$C(n)=C(n,\frac{n}{n+1})$ and $s>0$ is a free parameter.
\end{lemma}

\begin{proof}For all $g\in C^1(M)$, $g\geq0$ satisfying (\ref{cond-1}) and (\ref{cond-2}), Lemma
\ref{Sobo-ineq-HS} implies
\begin{equation}\label{ineq-g}||g||_{\frac{n}{n-1}}\leq
C(n)\int_M(|\nabla g|+g|H|)d\mu.
\end{equation}
Substituting $g=f^{\frac{2(n-1)}{n-2}}$ into (\ref{ineq-g}) gives
$$\left(\int_Mf^{\frac{2n}{n-2}}d\mu\right)^{\frac{n-1}{n}}\leq\frac{2(n-1)}{n-2}C(n)
\int_Mf^{\frac{n}{n-2}}|\nabla
f|d\mu+C(n)\int_MHf^{\frac{2(n-1)}{n-2}}d\mu.$$ By H\"{o}lder's
inequality, we get

\begin{equation*}
\begin{split}
\bigg(\int_{M}f^{\frac{2n}{n-2}}d\mu\bigg)^{\frac{n-1}{n}}\leq&
C(n)\bigg[\frac{2(n-1)}{n-2}\bigg(\int_Mf^{\frac{2n}{n-2}}d\mu\bigg)^{\frac{1}{2}}\bigg(\int_M|\nabla
f|^2d\mu\bigg)^{\frac{1}{2}}\\
&+\bigg(\int_M
H_0^2f^2d\mu\bigg)^{\frac{1}{2}}\bigg(\int_Mf^{\frac{2n}{n-2}}d\mu\bigg)^{\frac{1}{2}}\bigg]
\end{split}
\end{equation*}
Then
\begin{equation*}
\bigg(\int_{M}f^{\frac{2n}{n-2}}d\mu\bigg)^{\frac{n-2}{2n}}\leq
C(n)\bigg[\frac{2(n-1)}{n-2}\bigg(\int_M|\nabla
f|^2d\mu\bigg)^{\frac{1}{2}} +\bigg(\int_M
H_0^2f^2d\mu\bigg)^{\frac{1}{2}}\bigg].
\end{equation*}

This implies \begin{equation*}|| f||^2_{\frac{2n}{n-2}}\leq
C^2(n)\left[\frac{4(n-1)^2(1+s)}{(n-2)^2}||\nabla f||^2_2
+H^2_0\left(1+\frac{1}{s}\right)||f||^2_2\right],
\end{equation*} which is
desired.
\end{proof}

Now we establish an inequality involving the maximal value of the
squared norm of the mean curvature vector and its $L^{n+2}$-norm in
the space-time.
\begin{prop}\label{ineq-integral} Suppose that $F_t:M^n\rightarrow N^{n+d}$ $(n\geq3)$ is the
mean curvature flow solution for $t\in[0,T_0]$, where $N$ has
bounded geometry. Then
$$\max_{(x,t)\in M\times [\frac{T_0}{2},T_0]} |H|^2(x,t)\leq C\left(\int^{T_0}_0\int_{M_t}
 |H|^{n+2}d\mu_t dt\right)^{\frac{2}{n+2}},$$
where $C$ is a constant depending only on $n$, $T_0$,
$\sup_{(x,t)\in M\times[0,T_0]}|A|$, $K_1$, $K_2$ and $i_N$.
\end{prop}

\begin{proof} In the following proof, we always denote by $C$ the
constant depending on some quantities. We make use of Moser
iteration for parabolic equations.  Here we follow the computation
in \cite{DWY}. From the evolution equation of the second fundamental
form  in Lemma \ref{evo-equ}, we have the following differential
inequality.
\begin{equation}\label{ineq-H2}\frac{\partial}{\partial t}|H|^2\leq\triangle |H|^2+\beta
|H|^2, \end{equation}
where $\beta$ is a positive constant depending
only on $n$, $\sup_{(x,t)\in M\times[0,T_0]}|A|$, $K_1$ and $K_2$.
For $0<R<R'<\infty$ and $x\in M$, we set
\[\eta=\left\{ \begin{array}{ll}
1&\ \ \ \ \ \ x\in B_{g(0)}(x,R),\\
\eta\in [0,1]\ and \ |\nabla\eta|_{g(0)}\leq\frac{1}{R'-R}&\ \ \ \ \ \ x\in B_{g(0)}(x,R')\setminus B_{g(0)}(x,R),\\
0&\ \ \ \ \ \ x\in M\setminus B_{g(0)}(x,R').
\end{array}
\right.\] Since $supp\ \eta\subseteq B_{g(0)}(x,R')$, we assume that
$R'$ is sufficiently small such that  $\eta$ satisfies
(\ref{cond-1}) and (\ref{cond-2}) with respect to $g(0)$. On the
other hand, the area of some fixed subset in $M$ is non-increasing
along the mean curvature flow, hence $\eta$ satisfies (\ref{cond-1})
and (\ref{cond-2}) with respect to each $g(t)$ for $t\in [0,T_0]$.
Putting $f=|H|^2$ and $B(R')=B_{g(0)}(x,R')$, the inequality
(\ref{ineq-H2}) implies that, for any $q\geq2$,
\begin{equation}\label{ineq-1}
\begin{split}
\frac{1}{q}\frac{\partial}{\partial t}\int_{B(R')}f^q\eta^2d\mu_t
&\leq\int_{B(R')}\left(\eta^2 f^{q-1}\triangle fd\mu_t +\beta
f^q\eta^2\right)d\mu_t
+\int_{B(R')}\frac{1}{q}f^q\eta^2\frac{\partial}{\partial
t}d\mu_t\\
&=\int_{B(R')}\left(\eta^2 f^{q-1}\triangle fd\mu_t +\beta
f^q\eta^2\right)d\mu_t-\int_{B(R')}\frac{1}{q}f^{q+1}\eta^2d\mu_t\\
&\leq\int_{B(R')}\left(\eta^2 f^{q-1}\triangle fd\mu_t +\beta
f^q\eta^2\right)d\mu_t.
\end{split}
\end{equation}
Here we have used the evolution equation of the volume form in Lemma
\ref{evo-equ}. Integrating  by parts we obtain
\begin{equation}\label{ineq-2}
\begin{split}
\int_{B(R')}\eta^2 f^{q-1}\triangle fd\mu_t
=&-\frac{4(q-1)}{q^2}\int_{B(R')}|\nabla(f^{\frac{q}{2}}\eta)|^2d\mu_t
+\frac{4}{q^2}\int_{B(R')}|\nabla\eta|^2f^q d\mu_t\\
&+\frac{4(q-2)}{q^2}\int_{B(R')}\langle\nabla(f^{\frac{q}{2}}\eta),f^{\frac{q}{2}}\nabla\eta\rangle d\mu_t\\
\leq&-\frac{2}{q}\int_{B(R')}|\nabla(f^{\frac{q}{2}}\eta)|^2d\mu_t+\frac{2}{q}\int_{B(R')}|\nabla\eta|^2f^qd\mu_t.
\end{split}
\end{equation}

Thus by (\ref{ineq-1}) and (\ref{ineq-2}) we obtain
\begin{equation}
\begin{split}
\frac{1}{q}\frac{\partial}{\partial t}\int_{B(R')} f^q\eta^2d\mu_t
\leq&-\frac{2}{q}\int_{B(R')}|\nabla(f^{\frac{q}{2}}\eta)|^2d\mu_t\\
&+\beta\int_{B(R')}f^q\eta^2d\mu_t+\frac{2}{q}\int_{B(R')}|\nabla\eta|^2f^qd\mu_t.
\end{split}\end{equation}
 This implies
\begin{equation}\label{ineq-3}
\begin{split}\frac{\partial}{\partial t}\int_{B(R')}
f^q\eta^2d\mu_t&+\int_{B(R')}|\nabla(f^{\frac{q}{2}}\eta)|^2d\mu_t\\
&\leq2\int_{B(R')}|\nabla\eta|^2f^qd\mu_t +\beta
q\int_{B(R')}f^q\eta^2d\mu_t. \end{split}\end{equation}
For any $0<\tau<\tau'<T_0$, define a function $\psi$ on $[0,T_0]$ by\\
\[\psi(t)=\left\{ \begin{array}{ll}
0&\ \ \ \ \ \ 0\leq t\leq \tau,\\
\frac{t-\tau}{\tau'-\tau}&\ \ \ \ \ \ \tau\leq t\leq \tau',\\
1&\ \ \ \ \ \ \tau'\leq t\leq T_0.
\end{array}
\right.\]
 Then from (\ref{ineq-3}) we get
\begin{equation}\label{ineq-4}
\begin{split}
\frac{\partial}{\partial t}\left(\psi\int_{B(R')}
f^q\eta^2d\mu_t\right)d\mu_t+&\psi\int_{B(R')}|\nabla(f^{\frac{q}{2}}\eta)|^2d\mu_t\\
&\leq 2\psi\int_{B(R')}|\nabla\eta|^2f^qd\mu_t+(\beta
q\psi+\psi')\int_{B(R')} f^q\eta^2d\mu_t.
\end{split}\end{equation}

For any $t\in [\tau',T_0]$, integrating both sides of (\ref{ineq-4})
on $[\tau,t]$ implies
\begin{equation}\label{ineq-5}
\begin{split}
\int_{B(R')}
f^q\eta^2d\mu_t&+\int^t_{\tau'}\int_{B(R')}|\nabla(f^{\frac{q}{2}}\eta)
|^2d\mu_tdt\\
&\leq2\int^{T_0}_\tau\int_{B(R')}|\nabla\eta|^2f^qd\mu_tdt
+\left(\beta
q+\frac{1}{\tau'-\tau}\right)\int^{T_0}_\tau\int_{B(R')}
f^q\eta^2d\mu_tdt. \end{split}\end{equation}

By the Sobolev inequality in Lemma \ref{Sobo-ineq}, we obtain
\begin{equation}\label{ineq-5'}
\begin{split}
\left(\int_{B(R')}
f^{\frac{qn}{n-2}}\eta^{\frac{2n}{n-2}}d\mu_t\right)^{\frac{n-2}{n}}=&||
f^{\frac{q}{2}}\eta||^2_{\frac{2n}{n-2}}\\
\leq&\frac{4(n-1)^2(1+s)C^2(n)}{(n-2)^2}||\nabla(f^{\frac{q}{2}}\eta)||^2_2\\
&+CC^2(n)\Big(1+\frac{1}{s}\Big)|| f^{\frac{q}{2}}\eta||^2_2,
\end{split}\end{equation}
where $C$ depends on $n$ and $\sup_{(x,t)\in M\times[0,T_0]}|A|$.
Combining  (\ref{ineq-5}) and (\ref{ineq-5'}) implies that
\begin{equation}\label{ineq-6}
\begin{split}
&\int^{T_0}_{\tau'}\int_{B(R')}
f^{q(1+\frac{2}{n})}\eta^{2+\frac{1}{n}}d\mu_tdt\\
\leq&\int^{T_0}_{\tau'}\left(\int_{B(R')}
f^{q}\eta^2d\mu_t\right)^{\frac{2}{n}}\left(\int_{B(R')}
f^{\frac{nq}{n-2}}\eta^{\frac{2n}{n-2}}
\mu_t\right)^{\frac{n-2}{n}}dt\\
\leq&\max_{t\in [\tau',T_0]}\left(\int_{B(R')}
f^{q}\eta^2d\mu_t\right)^{\frac{2}{n}}\times\int_\tau^{T_0}
\Big[\frac{4(n-1)^2(1+s)C^2(n)}{(n-2)^2}||\nabla(f^{\frac{q}{2}}\eta)||^2_2\\
&+CC^2(n)\Big(1+\frac{1}{s}\Big)||
f^{\frac{q}{2}}\eta||^2_2\Big]dt\\
\leq&C\max_{t\in [\tau',T_0]}\left(\int_{B(R')}
f^{q}\eta^2d\mu_t\right)^{\frac{2}{n}}\times\int_\tau^{T_0}\Big[||\nabla(f^{\frac{q}{2}}\eta)||^2_2+||
f^{\frac{q}{2}}\eta||^2_2\Big]dt\\
\leq&
C\bigg[2\int^{T_0}_{\tau}\int_{B(R')}|\nabla\eta|^2f^qd\mu_tdt\\
&+\Big(\beta q+\frac{1}{\tau'-\tau}\Big)\int^{T_0}_\tau\int_{B(R')}
f^q\eta^2d\mu_tdt\bigg]^{1+\frac{2}{n}},
\end{split}\end{equation}
where we have put $s=1$ and $C$ is a constant depending only on $n$
and $\sup_{(x,t)\in M\times[0,T_0]}|A|$.

Note that $|\nabla\eta|_{g(t)}\leq|\nabla\eta|^2_{g(0)}e^{lt}$,
where $l=\max_{0\leq t\leq T_0}||\frac{\partial g}{\partial
t}||_{g(t)}$. Thus
\begin{equation*}
\begin{split}\int^{T_0}_{\tau}\int_{B(R')}|\nabla\eta|^2f^qd\mu_tdt\leq&
\int^{T_0}_{\tau}\int_{B(R')}|\nabla\eta|^2_{g(0)}e^{lt}f^qd\mu_tdt\\
\leq&
\frac{e^{CT_0}}{(R'-R)^2}\int^{T_0}_{\tau}\int_{B(R')}f^qd\mu_tdt,
\end{split}
\end{equation*} for some positive  constant $C$  depending on $n$ and
$\sup_{(x,t)\in M\times [0,T_0]}|A|$. This together with
(\ref{ineq-6}) implies that
\begin{equation} \label{ineq-7}\begin{split}
\int^{T_0}_{\tau}\int_{B(R)} f^{q(1+\frac{2}{n})}d\mu_tdt\leq&
C\left(\beta q+\frac{1}{\tau'-\tau}+\frac{2e^{CT_0}}{(R'-R)^2}\right)^{1+\frac{2}{n}}\\
&\times\left(\int^{T_0}_\tau\int_{B(R')}
f^qd\mu_tdt\right)^{1+\frac{2}{n}},
\end{split}\end{equation} where
$C$ is a positive constant depending on $n$ and $\sup_{(x,t)\in
M\times[0,T_0]}|A|$.

Putting $L(q,t,R)=\int^{T_0}_{t}\int_{B(R)}f^{q}d\mu_tdt$, we
have from (\ref{ineq-7})
\begin{equation}
\label{ineq-8}L\Big(q\Big(1+\frac{2}{n}\Big),\tau',R\Big)\leq
C\left(\beta
q+\frac{1}{\tau'-\tau}+\frac{2e^{CT_0}}{(R'-R)^2}\right)^{1+\frac{2}{n}}
L(q,\tau,R')^{1+\frac{2}{n}}.\end{equation} We set
\begin{equation*}\mu=1+\frac{2}{n},\ q_k=\frac{n+2}{2}\mu^k,\
\tau_k=\Big(1-\frac{1}{\mu^{k+1}}\Big)t, \
R_k=\frac{R'}{2}(1+\frac{1}{\mu^{k/2}}).\end{equation*} Then it
follows from (\ref{ineq-8}) that
\begin{equation}\begin{split}&L(q_{k+1},\tau_{k+1},R_{k+1})^{\frac{1}{q_{k+1}}}\\
\leq& C^{\frac{1}{q_{k+1}}}
\Big[\frac{(n+2)\beta}{2}+\frac{\mu^2}{\mu-1}\cdot\frac{1}{t}+\frac{4e^{CT_0}}{R'^2}\cdot\frac{\mu}
{(\sqrt{\mu}-1)^2}\Big]^{\frac{1}{q_k}}\\
&\times\mu^{\frac{k}{q_{k}}}L(q_{k},\tau_{k},R_k)^{\frac{1}{q_{k}}}.\nonumber
\end{split}\end{equation} Hence
\begin{equation}\begin{split}\label{aaaa}
&L(q_{m+1},\tau_{m+1},R_{m+1})^{\frac{1}{q_{m+1}}}\\
\leq& C^{\sum^m_{k=0}\frac{1}{q_{k+1}}}
\Big[\frac{(n+2)\beta}{2}+\frac{\mu^2}{\mu-1}\cdot\frac{1}{t}+\frac{4e^{CT_0}}{R'^2}\cdot
\frac{\mu}{(\sqrt{\mu}-1)^2}\Big]^{\sum^m_{k=0}\frac{1}{q_{k}}}\\
&\times\mu^{\sum^m_{k=0}\frac{k}{q_{k}}}L(q_{0},\tau_{0},R_0)^{\frac{1}{q_{0}}}.
\end{split}\end{equation}
 As $m\rightarrow+\infty$, we conclude from (\ref{aaaa}) that
\begin{equation}\label{ineq-9}f(x,t)\leq
C^{\frac{n}{n+2}}\Big(C+\frac{1}{t}+\frac{e^{CT_0}}{R'^2}\Big)\Big(1+\frac{2}{n}\Big)^{\frac{n}{2}}
\left(\int^{T_0}_0\int_{M_t}f^{\frac{{n+2}}{2}}d\mu_tdt\right)^{\frac{2}{n+2}},
\end{equation} for some positive constant $C$
depending on $n$, $\sup_{M\times [0,T]}$, $K_1$ and $K_2$.

Note that we choose $R'$ sufficient small such that
\begin{equation}\label{cond-3}K_2(n+1)^{\frac{2}{n}}(\omega_n^{-1}Vol_{g(0)}(B(R'))^{\frac{2}{n}}\leq1,\end{equation}
and
\begin{equation}\label{cond-4}2K_2^{-\frac{1}{2}}\sin^{-1}K_2^{\frac{1}{2}}(n+1)^{\frac{1}{n}}
(\omega_n^{-1}Vol_{g(0)}(B(R'))^{\frac{1}{n}}\leq i_N.\end{equation}
For $g(0)$, there is a non-positive constant $K$ depending on $n$,
$\max_{x\in M_0}|A|$, $K_1$ and $K_2$ such that the sectional
curvature of $M_0$ is bounded from below by $K$. By the
Bishop-Gromov volume comparison theorem, we have
$$Vol_{g(0)}(B(R'))\leq Vol_K(B(R')),$$ where $Vol_K(B(R'))$ is the
volume of a ball with radius $R'$ in the $n$-dimensional complete
simply connected space form with constant curvature $K$. Let $R'$ be
the largest number such that
\begin{equation}K_2(n+1)^{\frac{2}{n}}(\omega_n^{-1}Vol_K(B(R'))^{\frac{2}{n}}\leq1,\nonumber\end{equation}
and
\begin{equation}2K_2^{-\frac{1}{2}}\sin^{-1}K_2^{\frac{1}{2}}(n+1)^{\frac{1}{n}}
(\omega_n^{-1}Vol_K(B(R'))^{\frac{1}{n}}\leq
i_N.\nonumber\end{equation}
 Then $R'$ depends only on $n$, $K_1$,
$K_2$, $i_N$ and $\sup_{(x,t)\in M\times[0,T_0]}|A|$, and
$Vol_{g(0)}(B(R'))$ satisfies (\ref{cond-3}) and (\ref{cond-4}).
This together with (\ref{ineq-9}) implies
$$\max_{(x,t)\in M\times[\frac{T_0}{2},T_0]}|H|^2(x,t)
\leq C\left(\int^{T_0}_0\int_{M_t}|
H|^{n+2}d\mu_tdt\right)^{\frac{2}{n+2}},$$ where $C$ is a constant
depending on $n$, $T_0$ and $\sup_{(x,t)\in M\times[0,T_0]}|A|$, $K_1$,
$K_2$ and $i_N$.
\end{proof}

Now we give a sufficient condition that assures the extension of the
mean curvature flow of submanifolds in a Riemannian manifold.
\begin{theorem}\label{extension}
Let $F_t:M^n\rightarrow N^{n+d}$ $(n\geq3)$ be the mean curvature
flow solution of closed submanifolds in a finite time interval $[0,T)$. If\\
(i) there exist positive constants $a$ and $b$ such that $|A|^2\leq a|H|^2+b$ for $t\in [0,T)$,\\
(ii) $||H||_{\alpha,M\times[0,T)}=\left(\int_0^T\int_{M_t}|H|^\alpha
d\mu_tdt\right)^{\frac{1}{\alpha}}<\infty$ for some
$\alpha\geq n+2$,\\
then this flow can be extended over time $T$.
\end{theorem}

\begin{proof}By H\"{o}lder's inequality, it is sufficient to prove the theorem for $\alpha=n+2$. We will
argue by contradiction.

Suppose that the solution of the mean curvature flow can't be
extended over $T$. Then the second fundamental form becomes
unbounded as $t\rightarrow T$. From assumption $(i)$, $|H|^2$ is
unbounded either.

First we choose an increasing time sequence $t^{(i)}$,
$i=1,2,\cdots$, such that $t^{(i)}\rightarrow T$ as $i\rightarrow
\infty$. Then we take a sequence of points $x^{(i)}\in M$ satisfying
\begin{equation*}|H|^2(x^{(i)},t^{(i)})=\max_{(x,t)\in M\times [0,t^{(i)}]}
|H|^2(x,t).\end{equation*} Put $$Q^{(i)}=|H|^2(x^{(i)},t^{(i)}),$$
then $Q^{(i)},i=1,2,\cdots$ is a nondecreasing sequence and
$\lim_{i\rightarrow\infty}Q^{(i)}=\infty.$ This together with
$\lim_{i\rightarrow \infty}t^{(i)}=T>0$ implies that there exists a
positive integer $i_0$ such that $Q^{(i)}t^{(i)}\geq1$ and
$Q^{(i)}\geq 1$ for $i\geq i_0$. Let $h$ be the Riemannian metric on
$N$. For $i\geq i_0$ and $t\in [0,1]$, we consider the rescaled mean
curvature flows
\begin{equation*}F^{(i)}(t)=F\left(\frac{t-1}{Q^{(i)}}+t^{(i)}\right):(M,g^{(i)}(t))\rightarrow
(N,Q^{(i)}h),\end{equation*} where
$g^{(i)}(t)=F^{(i)}(t)^*(Q^{(i)}h)$. Let $H_{(i)}$ and
$A^{(i)}=h^{(i)}_{jk}$  be the mean curvature vector and the second
fundamental form of $F^{(i)}(t)$ respectively. Then we have
\begin{equation}\label{ineq-Hi}|H_{(i)}|^2(x,t)\leq 1 \ \ on\ \ M\times [0,1].\end{equation}

From assumption $(i)$ again, inequality (\ref{ineq-Hi}) implies
$|A^{(i)}|\leq C$, where $C$ is a constant independent of $i$. Since
$(N,h)$ has bounded geometry and $Q^{(i)}\geq1$ for $i\geq i_0$,
$(N,Q^{(i)}h)$ also has bounded geometry for $i\geq i_0$ with the
same bounding constants as $(N,h)$. It follows from Proposition
\ref{ineq-integral} that for $i\geq i_0$
$$\max_{(x,t)\in M^{(i)}\times[\frac{1}{2},1]}|H_{(i)}|^2(x,t)
\leq C\left(\int^{1}_0\int_{M_t}| H_{(i)}|^{n+2}d\mu_{g^{(i)}(t)}
dt\right)^{\frac{2}{n+2}},$$ where $C$ is a constant independent of
$i$.

By \cite{CH}, there is a subsequence of pointed mean curvature flow
solutions
\begin{equation*}F^{(i)}(t):(M,g^{(i)}(t),x^{(i)})\rightarrow
(N,Q^{(i)}h),\ t\in[0,1]\end{equation*} that converges to a pointed
mean curvature flow solution
\begin{equation*}\widetilde{F}(t):(\widetilde{M},\widetilde{g}(t),\widetilde{x})\rightarrow
{R}^{n+d},\ t\in[0,1].\end{equation*} Denote by $\widetilde{H}$ the
mean curvature vector of $\widetilde{F}$, $t\in [0,1]$. Then we have
\begin{equation}\label{ineq-maxH}
\begin{split}\max_{(x,t)\in
\widetilde{M}\times[\frac{1}{2},1]}|\widetilde{H}|^2(x,t)\leq
&\lim_{i\rightarrow \infty}C\left(\int^{1}_0\int_{M_t}|
H_{(i)}|^{n+2}d\mu_{g^{(i)}(t)} dt\right)^{\frac{2}{n+2}}\\
\leq&\lim_{i\rightarrow
\infty}C\left(\int^{t^{(i)}+(Q^{(i)})^{-1}}_{t^{(i)}}\int_{M_t}|
H_{(i)}|^{n+2}d\mu dt\right)^{\frac{2}{n+2}}\\
=&0.
\end{split}
\end{equation}
The equality in (\ref{ineq-maxH}) holds because $\int^T_0\int_M
|H|^{n+2} d\mu_t dt<+\infty$ and $(Q^{(i)})^{-1}\rightarrow0$ as
$i\rightarrow\infty$. On the other hand, according to the choice of
the points, we have
\begin{equation*}|\widetilde{H}|^2(\widetilde{x},1)=\lim_{i\rightarrow \infty}|H_{(i)}|^2(x^{(i)},1)=1.\end{equation*}
This is a contradiction.
\end{proof}
\begin{remark} When $d=1$, Theorem
\ref{extension} generalizes the theorems in \cite{LS,XYZ1,XYZ2}.
 In fact, for $N^{n+1}=R^{n+1}$, we have the following computations.

(i) If $h_{ij}\geq -C$ for $(x,t)\in M\times [0,T)$ with some
$C\geq0$, let $\lambda_i$, $i=1,\cdots,n$ be the principal
curvatures. Then $\lambda_i+C\geq0$, which implies that
\begin{equation*}\sum_i(\lambda_i+C)^2\leq n(\sum_i(\lambda_i+C))^2\leq
2nH^2+2n^3C^2. \end{equation*} On the other hand,
\begin{equation*}\sum_i(\lambda_i+C)^2=|A|^2+2CH+nC^2\geq|A|^2-H^2+(n-1)C^2.
\end{equation*} Hence $|A|^2\leq (2n+1)H^2+(2n^3-n+1)C^2$ for $t\in
[0,T)$.

(ii) If $H>0$ at $t=0$,  then there exists a positive constant $C$
such that $|A|^2\leq CH^2$ at $t=0$. By \cite{H1}, we know that
$H>0$ for $t>0$ and \begin{equation*}\frac{\partial}{\partial
t}\left(\frac{|A|^2}{H^2}\right)=\triangle\left(\frac{|A|^2}{H^2}\right)+\frac{2}{H}
\left\langle\nabla
H,\nabla\left(\frac{|A|^2}{H^2}\right)\right\rangle-\frac{2}{H^4}|H\nabla_ih_{jk}-\nabla_iH\cdot
h_{jk}|^2.\end{equation*} By the maximum principle we obtain that
$|A|^2/H^2$ is uniformly bounded from above by its initial data.
Hence $|A|^2\leq CH^2$  for $t\in [0,T)$.

For  general $N^{n+1}$ with bounded geometry, we have similar
computations. Hence our Theorem \ref{extension} is a generalization.
\end{remark}

At the end of this section, we would like to propose the following
\begin{problem} Let $F_t:M\rightarrow N$ be the mean curvature flow solution
of closed submanifolds in a finite time interval $[0,T)$. Suppose
$||H||_{\alpha,M\times[0,T)}<\infty$ for some $\alpha\geq n+2$. Is
there a positive constant $\omega$ such that the solution exists in
$[0,T+\omega)$?
\end{problem}

\section{The convergence of mean curvature flow}

In this section we obtain some convergence theorems for the mean
curvature flow. The extension theorem proved in section 3 will be used to give
 a positive lower bound on the existence time of the mean curvature flow.

We need the following Sobolev inequality for submanifolds in the
Euclidean space.
\begin{lemma}\label{Sobolev-ineq-integral} Let $M$ be an $n(\geq3)$-dimensional closed
submanifold in $\mathbb{R}^{n + d}$. Then for all Lipschitz
functions $v$ on $M$, we have
\begin{equation*}\bigg(\int_M v^{\frac{2n}{n-2}}d\mu\bigg)^{\frac{n-2}{n}}\leq
C_{n}\bigg(\int_M|\nabla v|^2d\mu+\int_{M}|H|^{n+2}d\mu \int_M
v^2d\mu\bigg)
\end{equation*}
where $C_n$ is a positive constant depending only on $n$.
\end{lemma}
\begin{proof}The proof of the lemma   for $d=1$ was given in
\cite{LS}. Using the same method we can prove the lemma for $d>1$.
\end{proof}

Now we begin to prove the following convergence theorem for the mean
curvature flow.

\begin{theorem}\label{convergence}
Let $F_0:M^n\rightarrow R^{n+d}$ $(n\geq3)$ be a smooth closed
submanifold.  Then for any fixed $p>1$, there is a positive constant
$C_1$ depending on $n,p, Vol(M_0)$ and $||A||_{n+2}$, such that if
\begin{equation*}||\mathring{A}||_{p}<C_1,\end{equation*}
  then the mean curvature flow with $F_0$
as initial value has a unique solution $F:M\times[0,T)\rightarrow
R^{n+d}$ in a finite maximal time interval, and $F_t$ converges
uniformly to a point $x\in R^{n+d}$ as $t\rightarrow T$. The
rescaled maps $\widetilde{F}_t=\frac{F_t-x}{\sqrt{2n(T-t)}}$
converge in $C^{\infty}$ to a limiting embedding $\widetilde{F}_T$
such that $\widetilde{F}_T(M)$ is the unit $n$-sphere in some
$(n+1)$-dimensional subspace of $R^{n+d}$.
\end{theorem}

\begin{proof}
We set $\Lambda=||A||_{n+2}$. Denote by $T_{\max}$ the maximal
existence time of the mean curvature flow with $F_0$ as initial
value. It is easy to show that $T_{\max}<+\infty$ (see
\cite{WaM3} for a proof).

We split the proof to several steps.

\textbf{Step 1.} For any fixed positive number $\varepsilon$, we
first show that if
\begin{equation}
||\mathring{A}||_p<\varepsilon
\end{equation} for some $p>1$, then  $T_{\max}$  satisfies
$T_{\max}>T_0$ for some positive constant $T_0$ depending  on $n,\
p$, $\Lambda$ and independent of $\varepsilon$, and there hold
$||A(t)||_{n+2}<2\Lambda,$ $||\mathring{A}(t)||_p<2\varepsilon$ for
$t\in[0,T_0]$ .

 Put
\begin{equation*}T=\sup\{t\in [0,T_{\max}):
||A(t)||_{n+2}<2\Lambda, ||\mathring{A}(t)||_p<2\varepsilon\}.
\end{equation*}
We consider the mean curvature flow on the time interval $[0,T)$.

By the definition of $T$ we have $\int_{M_t}|A|^{n+2}d\mu_t\leq
(2\Lambda)^{n+2}$ for $t\in [0,T)$. From Lemma
\ref{Sobolev-ineq-integral} we have for a Lipschitz function $v$,
\begin{equation}\label{sobo-Mt}\bigg(\int_M v^{\frac{2n}{n-2}}d\mu\bigg)^{\frac{n-2}{n}}\leq
C_{n}\bigg(\int_M|\nabla
v|^2d\mu+n^{\frac{n+2}{2}}(2\Lambda)^{n+2}\int_M v^2d\mu\bigg).
\end{equation}

From (\ref{evo-A}), we  have
\begin{equation*}
\frac{\partial}{\partial t} |A|^2\leq \Delta |A|^2 +c_1|A|^4,
\end{equation*} for some positive constant $c_1$ depending only on
$n$.
Putting $u=|A|^2$, we have
\begin{eqnarray}\label{evo-f}
\frac{\partial}{\partial t}u\leq \Delta u+c_1u^2.
\end{eqnarray}
 From (\ref{evo-f})
and (\ref{evo-d mu}) we have
\begin{equation}\label{evo-int-f^p}
\begin{split}
\frac{\partial}{\partial t}\int_{M_t} u^{\frac{n+2}{2}}d\mu_t
=&\int_{M_t}\frac{n+2}{2}u^{\frac{n+2}{2}-1}\frac{\partial}{\partial
t}ud\mu_t
+\int_{M_t} u^{\frac{n+2}{2}} \frac{\partial}{\partial t}d\mu_t\\
=& \frac{n+2}{2}\int_{M_t} u^{\frac{n+2}{2}-1}(\Delta
u+cu^2)d\mu_t-\int_{M_t}H^2u^{\frac{n+2}{2}}d\mu_t\\
\leq&-\frac{4n}{n+2}\int_{M_t}|\nabla
u^{\frac{n+2}{4}}|^2d\mu_t+\frac{n+2}{2}c_1\int_{M_t}
u^{\frac{n+2}{2}+1}d\mu_t.
\end{split}
\end{equation}
For the second term of the right hand side of (\ref{evo-int-f^p}),
we have by H\"{o}lder's inequality

\begin{equation}\label{ineq-fp+1-1}
\begin{split}
\int_{M_t} u^{\frac{n+2}{2}+1}d\mu_t \leq&
\bigg(\int_{M_t}u^{\frac{n+2}{2}}d\mu_t\bigg)^{{\frac{2}{n+2}}}\cdot
\bigg(\int_{M_t}(u^{\frac{n+2}{2}})^{{\frac{n+2}{n}}}d\mu_t\bigg)^{{\frac{n}{n+2}}}\\
\leq&
\bigg(\int_{M_t}u^{\frac{n+2}{2}}d\mu_t\bigg)^{{\frac{2}{n+2}}}\cdot
\bigg(\int_{M_t}u^{\frac{n+2}{2}}d\mu_t\bigg)^{\frac{2}{n+2}}\cdot
\bigg(\int_{M_t}(u^{\frac{n+2}{2}})^{\frac{2n}{n-2}}d\mu_t\bigg)^{\frac{n-2}{n+2}}\\
\leq&
\bigg(\int_{M_t}u^{\frac{n+2}{2}}d\mu_t\bigg)^{{\frac{2}{n+2}}}\cdot
\bigg(\int_{M_t}u^{\frac{n+2}{2}}d\mu_t\bigg)^{\frac{2}{n+2}}\\
&\times\bigg[C_n\bigg(\int_{M_t}|\nabla
u^{\frac{n+2}{4}}|^2d\mu_t+\int_{M_t}|H|^{n+2}d\mu_t
\int_{M_t}u^\frac{n+2}{2}d\mu_t\bigg)\bigg]^{\frac{n}{n+2}}\\
\leq&
\bigg(\int_{M_t}u^{\frac{n+2}{2}}d\mu_t\bigg)^{{\frac{4}{n+2}}}
\cdot\bigg[C_n^{\frac{n}{n+2}}\bigg(\int_{M_t}|\nabla
u^{\frac{n+2}{4}}|^2d\mu_t\bigg)^{\frac{n}{n+2}}\\
&+n^{\frac{n}{2}}(2\Lambda)^nC_n^{\frac{n}{n+2}}\bigg(\int_{M_t}u^{\frac{n+2}{2}}d\mu_t\bigg)^{\frac{2n}{n+2}}
\bigg]\\
\leq&n^{\frac{n}{2}}(2\Lambda)^nC_n^{\frac{n}{n+2}}\bigg(\int_{M_t}u^{\frac{n+2}{2}}d\mu_t\bigg)^{2}+C_n^{\frac{n}{n+2}}\cdot
\frac{2}{n+2}\epsilon^{\frac{n+2}{2}}\bigg(\int_{M_t}u^{\frac{n+2}{2}}d\mu_t\bigg)^{2}\\
&+C_n^{\frac{n}{n+2}}\cdot\frac{n}{n+2}\epsilon^{-\frac{n+2}{n}}\int_{M_t}|\nabla
u^{\frac{n+2}{4}}|^2d\mu_t,
\end{split}
\end{equation}
for any $\epsilon>0$. Combining (\ref{evo-int-f^p}) and
(\ref{ineq-fp+1-1}), we have
\begin{equation}\label{ineq-fp}
\begin{split}
\frac{\partial}{\partial t}\int_{M_t} u^{\frac{n+2}{2}}d\mu_t \leq&
\frac{n+2}{2}c_1\Big(n^{\frac{n}{2}}(2\Lambda)^nC_n^{\frac{n}{n+2}}+C_n^{\frac{n}{n+2}}\cdot
\frac{2}{n+2}\epsilon^{\frac{n+2}{2}}\Big)\bigg(\int_{M_t}u^{\frac{n+2}{2}}d\mu_t\bigg)^{2}\\
&+\Big(\frac{n}{2}c_1C_n^{\frac{n}{n+2}}\epsilon^{-\frac{n+2}{n}}-\frac{4n}{n+2}\Big)\int_{M_t}|\nabla
u^{\frac{n+2}{4}}|^2d\mu_t.
\end{split}
\end{equation}
Picking
$\epsilon=\bigg(\frac{n(n+2)c_1C_n^{\frac{n}{n+2}}}{8}\bigg)^{\frac{n}{n+2}}$,
inequality (\ref{ineq-fp}) reduces to
\begin{equation}\label{ineq-fp'}
\begin{split}
\frac{\partial}{\partial t}\int_{M_t} |A|^{{n+2}}d\mu_t \leq
c_2\bigg(\int_{M_t}|A|^{n+2}d\mu_t\bigg)^{2},
\end{split}
\end{equation}
where
$c_2=\frac{n+2}{2}c_1\bigg(n^{\frac{n}{2}}(2\Lambda)^nC_n^{\frac{n}{n+2}}+C_n^{\frac{n}{n+2}}\cdot
\frac{2}{n+2}\bigg(\frac{n(n+2)c_1C_n^{\frac{n}{n+2}}}{8}\bigg)^{\frac{n}{2}}\bigg)$.

From (\ref{ineq-fp'}), we see by the maximal principle that, for
$t\in [0,\min\{T,T_1\})$, where
$T_1=\frac{1-(\frac{2}{3})^{n+2}}{c_2 \Lambda^{n+2}}$, there holds
\begin{equation}\label{2222}
||A(t)||_{n+2}<\frac{3}{2}\Lambda.
\end{equation}

Now we consider the evolution equation of $|\mathring{A}|^2$. By a simple
computation, we have
\begin{equation}\label{evo-phi2} \begin{split}
\frac{\partial}{\partial t}|\mathring{A}|^2\leq&\Delta |\mathring{A}|^2
-2|\nabla\mathring{A}|^2+c_3|A|^2|\mathring{A}|^2,
\end{split}\end{equation}
where $c_3\geq c_1$ is a positive constant depending only on $n$.

 Define a
tensor $\tilde{\mathring{A}}$ by
$\tilde{\mathring{A}}^{\alpha}_{ij}=\mathring{A}^{\alpha}_{ij}+\sigma\eta^\alpha
\delta_{ij}$, where $\eta^\alpha=1$. Set
$h_\sigma=|\tilde{\mathring{A}}|=(|\mathring{A}|^2+nd\sigma^2)^{\frac{1}{2}}$.
Then from (\ref{evo-phi2})  we have
\begin{equation}\label{evo-h}
\frac{\partial}{\partial t}h_\sigma\leq\Delta h_\sigma
+c_3|A|^2h_\sigma.
\end{equation}

For any $r\geq p>1$, we have
\begin{equation}
\begin{split}\label{ineq-h}
\frac{1}{r}\frac{\partial}{\partial t}\int_{M_t} h_\sigma^{r}d\mu_t
=&\int_{M_t}h_\sigma^{r-1}\frac{\partial}{\partial t}h_\sigma d\mu_t
+\frac{1}{r}\int_{M_t} h_\sigma^{p} \frac{\partial}{\partial t}d\mu_t\\
\leq&-\frac{4(r-1)}{r^2}\int_{M_t}|\nabla h_\sigma
^{\frac{r}{2}}|^2d\mu_t+c_3\int_{M_t}|A|^2h_\sigma^rd\mu_t.
\end{split}
\end{equation}

For the second term of the right hand side of (\ref{ineq-h}), we
have the following estimate.

\begin{equation}
\begin{split}\label{second-term}
\int_{M_t}|A|^2h_\sigma^rd\mu_t
\leq&\bigg(\int_{M_t}|A|^{n+2}d\mu_t\bigg)^{\frac{2}{n+2}}\cdot\bigg(\int_{M_t}h_\sigma^{r\cdot
\frac{n+2}{n}}d\mu_t\bigg)^{\frac{n}{n+2}}\\
\leq&(2\Lambda)^{2}\bigg(\int_{M_t}h_\sigma^r
d\mu_t\bigg)^{\frac{2}{n+2}}
\cdot\bigg(\int_{M_t}(h_\sigma^r)^{\frac{n}{n-2}}
d\mu_t\bigg)^{\frac{n-2}{n}\cdot \frac{n}{n+2}}\\
\leq &(2\Lambda)^{2}\bigg(\int_{M_t}h_\sigma^r
d\mu_t\bigg)^{\frac{2}{n+2}}\cdot\bigg[C_{n}\bigg(\int_M|\nabla
h_\sigma^{\frac{r}{2}}|^2d\mu_t\\
&+n^{\frac{n+2}{2}}(2\Lambda)^{n+2}\int_M
h_\sigma^rd\mu_t\bigg)\bigg]^{\frac{n}{n+2}}\\
\leq&(2\Lambda)^{2}\bigg(\int_{M_t}h_\sigma^r
d\mu_t\bigg)^{\frac{2}{n+2}}\cdot\bigg[C_{n}^{\frac{n}{n+2}}\bigg(\int_M|\nabla
h_\sigma^{\frac{r}{2}}|^2d\mu_t\bigg)^{\frac{n}{n+2}}\\
&+n^{\frac{n}{2}}(2\Lambda)^{n}C_{n}^{\frac{n}{n+2}}\bigg(\int_M
h_\sigma^rd\mu_t\bigg)^{\frac{n}{n+2}}\bigg]\\
=&n^{\frac{n}{2}}(2\Lambda)^{n+2}C_{n}^{\frac{n}{n+2}}\int_M
h_\sigma^rd\mu_t\\
&+(2\Lambda)^{2}C_{n}^{\frac{n}{n+2}}\bigg(\int_{M_t}h_\sigma^r
d\mu_t\bigg)^{\frac{2}{n+2}}\cdot\bigg(\int_M|\nabla
h_\sigma^{\frac{r}{2}}|^2d\mu_t\bigg)^{\frac{n}{n+2}}\\
\leq&n^{\frac{n}{2}}(2\Lambda)^{n+2}C_{n}^{\frac{n}{n+2}}\int_M
h_\sigma^rd\mu_t\\
&+(2\Lambda)^{2}C_{n}^{\frac{n}{n+2}}\cdot\frac{2}{n+2}\mu^{\frac{n+2}{2}}\int_{M_t}h_\sigma^r
d\mu_t\\
&+(2\Lambda)^{2}C_{n}^{\frac{n}{n+2}}\cdot\frac{n}{n+2}\mu^{-\frac{n+2}{n}}\int_M|\nabla
h_\sigma^{\frac{r}{2}}|^2d\mu_t,
\end{split}
\end{equation}
for any $\mu>0$. Therefore, combining (\ref{ineq-h}) and
(\ref{second-term}) we have

\begin{equation}
\begin{split}\label{ineq-h1}
\frac{1}{r}\frac{\partial}{\partial t}\int_{M_t} h_\sigma^{r}d\mu_t
\leq&\bigg(c_3(2\Lambda)^{2}C_{n}^{\frac{n}{n+2}}\cdot\frac{n}{n+2}\mu^{-\frac{n+2}{n}}-\frac{4(r-1)}{r^2}\bigg)\int_{M_t}|\nabla
h_\sigma
^{\frac{r}{2}}|^2d\mu_t\\
&+c_3\bigg(n^{\frac{n}{2}}(2\Lambda)^{n+2}C_{n}^{\frac{n}{n+2}}+(2\Lambda)^{2}C_{n}^{\frac{n}{n+2}}\cdot\frac{2}{n+2}\mu^{\frac{n+2}{2}}\bigg)\int_M
h_\sigma^rd\mu_t.
\end{split}
\end{equation}
Choose $\mu=\Big(\frac{c_4r^2p}{3rp-4p+r}\Big)^{\frac{n}{n+2}}$,
where
$c_4=c_3(2\Lambda)^{2}C_{n}^{\frac{n}{n+2}}\cdot\frac{n}{n+2}$.
Then from (\ref{ineq-h1}), we have
\begin{equation}\label{ineq-h22}
\frac{\partial}{\partial t}\int_{M_t}
h_\sigma^{r}d\mu_t+\Big(1-\frac{1}{p}\Big)\int_{M_t}|\nabla h_\sigma
^{\frac{r}{2}}|^2d\mu_t\leq\bigg(c_5+c_6\Big(\frac{r^2p}{3rp-4p+r}\Big)^{\frac{n}{2}}\bigg)\cdot
r\cdot\int_M h_\sigma^rd\mu_t,
\end{equation}
where $c_5=c_3(2\Lambda)^2C_{n}^{\frac{n}{n+2}}$ and
$c_6=c_3n^{\frac{n}{2}}(2\Lambda)^{n+2}C_{n}^{\frac{n}{n+2}}\cdot\frac{2}{n+2}\cdot
c_4^{\frac{n}{2}}$.

Let $r=p$. Then (\ref{ineq-h2}) reduces to
\begin{equation}\label{ineq-h3}
\frac{\partial}{\partial t}\int_{M_t} h_\sigma^{p}d\mu_t\leq
c_7\int_M h_\sigma^pd\mu_t,
\end{equation}\
where
$c_7=\bigg(c_5+c_6\Big(\frac{p^2}{3p-3}\Big)^{\frac{n}{2}}\bigg)\cdot
p$.

Letting $\sigma\rightarrow0$,  (\ref{ineq-h3}) becomes
\begin{equation*}
\frac{\partial}{\partial t}\int_{M_t} |\mathring{A}|^{p}d\mu_t\leq
c_7\int_M |\mathring{A}|^pd\mu_t.
\end{equation*}
This implies by the maximal principle that, for $t\in
[0,\min\{T,T_2\})$, where $T_2=\frac{(n+2)\ln\frac{3}{2}}{c_7}$,
there holds
\begin{equation}\label{3333}
||\mathring{A}(t)||_p< \frac{3}{2}\varepsilon.
\end{equation}

Set $T_0=\min\{T_1,T_2\}.$   We claim that $T>T_0$. We prove this
claim by contradiction. Suppose that  $T\leq T_0$. Then (\ref{2222})
and (\ref{3333}) hold on $[0,T)$.

If $T<T_{\max}$,  from the smoothness of the mean curvature flow we
see that there exists a positive constant $\vartheta$ such that on
$[0,T+\vartheta)$ we have
\begin{equation*}
||A(t)||_{n+2}<\frac{5}{3}\Lambda,\ \ ||\mathring{A}(t)||_p<
\frac{5}{3}\varepsilon.
\end{equation*}
This contradicts to the definition of $T$.

If $T=T_{\max}$, we will show that the mean curvature flow can be
extended over time $T_{\max}$.

From  (\ref{ineq-h22}), we have

\begin{equation}\label{ineq-h2}
\frac{\partial}{\partial t}\int_{M_t}
h_\sigma^{r}d\mu_t+\Big(1-\frac{1}{p}\Big)\int_{M_t}|\nabla h_\sigma
^{\frac{r}{2}}|^2d\mu_t\leq c_8r^{n+1}\cdot\int_M h_\sigma^rd\mu_t,
\end{equation}
where
$c_8=\max\{\frac{c_5}{p^n},\frac{c_6}{(3p-3)^{\frac{n}{2}}}\}$.

 As  in the proof of Proposition
\ref{ineq-integral}, for any $\tau,\tau'$ such that
$0<\tau<\tau'<T_{\max}-\theta$, and  for any $t\in
[\tau',T_{\max}-\theta]$,  where $\theta$ is a small positive
constant,  we have from (\ref{ineq-h2})
\begin{equation}\label{ineq-h4}
\int_{M_t}
h_\sigma^{r}d\mu_t+\Big(1-\frac{1}{p}\Big)\int_{\tau'}^t\int_{M_t}|\nabla
h_\sigma ^{\frac{r}{2}}|^2d\mu_tdt \leq
\Big(c_8r^{n+1}+\frac{1}{\tau'-\tau}\Big)\int_{\tau}^{T_{\max}-\theta}\int_{M_t}
h_\sigma^rd\mu_tdt.
\end{equation}
 As in (\ref{ineq-6}),  we have by (\ref{sobo-Mt})
\begin{equation}\label{ineq-h5}
\begin{split}
&\int^{T_{\max}-\theta}_{\tau'}\int_{M_t}
h_\sigma^{(1+\frac{2}{n})}d\mu_tdt\\
\leq&\int^{T_{\max}-\theta}_{\tau'}\left(\int_{M_t}
h_\sigma^{r}d\mu_t\right)^{\frac{2}{n}}\cdot\left(\int_{M_t}
h_\sigma^{\frac{nr}{n-2}}d\mu_t
\right)^{\frac{n-2}{n}}dt\\
\leq&\max_{t\in [\tau',T_{max}-\theta]}\left(\int_{M_t}
h_\sigma^{r}d\mu_t\right)^{\frac{2}{n}}\cdot\int^{T_{\max}-\theta}_{\tau'}\left(\int_{M_t}
h_\sigma^{\frac{nr}{n-2}}d\mu_t \right)^{\frac{n-2}{n}}dt\\
\leq&C_n^{\frac{n-2}{n}}\cdot\max_{t\in
[\tau',T_{max}-\theta]}\left(\int_{M_t}
h_\sigma^{r}d\mu_t\right)^{\frac{2}{n}}\\
&\times \int^{T_{\max}-\theta}_{\tau'}\left(\int_{M_t}|\nabla
h_\sigma^{\frac{r}{2}}|^2d\mu_t+n^{\frac{n+2}{2}}(2\Lambda)^{n+2}\int_{M_t}
h_\sigma^rd\mu_t\right)dt.
\end{split}\end{equation}
From (\ref{ineq-h4}) and (\ref{ineq-h5}), we have
\begin{equation}\label{ineq-h6}
\begin{split}
\int^{T_{\max}-\theta}_{\tau'}\int_{M_t}
h_\sigma^{r(1+\frac{2}{n})}d\mu_tdt \leq&c_9\bigg(c_8
r^{n+1}+\frac{1}{\tau'-\tau}\bigg)^{1+\frac{2}{n}}\\
&\times
\bigg(\int^{T_{\max}-\theta}_{\tau}\int_{M_t}h_\sigma^rd\mu_tdt\bigg)^{1+\frac{2}{n}},
\end{split}\end{equation}
where $c_9=C_n^{\frac{n-2}{n}}\cdot
\max\{1,n^{\frac{n+2}{2}}(2\Lambda)^{n+2}T_0\cdot\frac{p}{p-1}\}$.

We put
\begin{equation*}J(r,t)=\int_t^{T_{\max}-\theta}\int_{M_t}h_\sigma^rd\mu_tdt.
\end{equation*}
Then from (\ref{ineq-h6}) we have
\begin{equation}\label{ineq-h7}J\Big(r\Big(1+\frac{2}{n}\Big),\tau'\Big)\leq c_9\bigg(c_8
r^{n+1}+\frac{1}{\tau'-\tau}\bigg)^{1+\frac{2}{n}}J(r,\tau)^{1+\frac{2}{n}}.
\end{equation}
We let \begin{equation*}\mu=1+\frac{2}{n},\ \ r_k=p\mu^k,\ \
\tau_k=\bigg(1-\frac{1}{\mu^{k+1}}\bigg)t.\end{equation*} Notice
that $\mu>1$. From (\ref{ineq-h7}) we have
\begin{equation*}
\begin{split}J(r_{k+1},\tau_{k+1})^{\frac{1}{r_{k+1}}} \leq
c_9^{\frac{1}{r_{k+1}}}\bigg(c_8p^{n+1} +\frac{\mu^2}{\mu-1}\cdot
\frac{1}{t}\bigg)^{\frac{1}{r_k}}\mu^{\frac{k}{r_k}\cdot
(n+1)}J(r_k,\tau_k)^{\frac{1}{r_{k}}}.
\end{split}\end{equation*}
Hence
\begin{equation*}
\begin{split}J(r_{m+1},\tau_{m+1})^{\frac{1}{r_{m+1}}} \leq&
c_9^{\sum_{k=0}^m\frac{1}{r_{k+1}}}\bigg(c_8p^{n+1}
+\frac{\mu^2}{\mu-1}\cdot
\frac{1}{t}\bigg)^{\sum_{k=0}^m\frac{1}{r_k}}\\
&\cdot\mu^{(n+1)\cdot\sum_{k=0}^m\frac{k}{r_k}}J(p,t)^{\frac{1}{p}}.
\end{split}\end{equation*}
As $m\rightarrow+\infty$, we conclude that
\begin{equation}\label{improt-estimate}
h_\sigma(x,t)\leq\bigg(1+\frac{2}{n}\bigg)^{\frac{n(n+1)(n+2)}{4p}}
c_9^{\frac{n}{2p}}\bigg(c_8p^{n+1}+\frac{(n+2)^2}{2nt}
\bigg)^{\frac{n+2}{2p}}\bigg(\int_{0}^{T_{\max}-\theta}\int_{M_t}h_\sigma^{p}d\mu_tdt\bigg)^{\frac{1}{p}}.
\end{equation}

Now let $\sigma\rightarrow 0$ and  $\theta\rightarrow 0$. Then we
have for $t\in [\frac{T_{\max}}{2},T_{\max})$,
\begin{equation*}
|\mathring{A}|^2(x,t)\leq C(n,p,\Lambda,\varepsilon,T_{\max})<+\infty.
\end{equation*}
This implies that
\begin{equation*}|A|^2\leq a|H|^2+b
\end{equation*}
on $[0,T_{\max})$ for some positive constants $a$ and $b$
independent of $t$. On the other hand, we also have
\begin{equation*}\int_{0}^{T_{\max}}\int_{M_t}|H|^{n+2}d\mu_tdt<+\infty,\end{equation*}
since $T_{\max}<+\infty$. Now we apply Theorem \ref{extension} to
conclude  that the mean curvature flow can be extended over time
$T_{\max}$. This is a contradiction. This completes the proof of the
claim.

By the definition of $T$, for $t\in [0,T_0]$, we also have
\begin{equation}\label{11111}||A(t)||_{n+2}<2\Lambda, \ \ \ ||\mathring{A}(t)||_p<2\varepsilon.\end{equation}
This completes Step 1.

\textbf{Step 2.} We denote by $Vol(\Sigma)$ the volume of a
Riemannian manifold $\Sigma$, and set $V=Vol(M_0)$. In this step we
show that if we choose $\varepsilon$ sufficiently small, then at
some time $T_3\in [\frac{T_0}{2},T_0]$, the mean curvature is
bounded from bellow by a positive constant depending on $n$, $p$,
$V$ and $\Lambda$.

 Since the area of the submanifold is non-increasing along the mean curvature flow, we see that for $t\in
[0,T_{\max})$, there holds
\begin{equation}\label{volume}Vol(M_t)\leq
V.\end{equation}

Since $M_t$ is a closed submanifold in the Euclidean space, by the
total mean curvature inequality (for the proof see \cite{chen}), we
have
\begin{equation*}n^n\omega_n\leq\int_{M_t}|H|^nd\mu_t\leq |H|^n_{\max}(t)Vol(M_t)\leq|H|^n_{\max}(t)V.\end{equation*}
Here $|H|_{\max}(t)=\max_{M_t}|H|(\cdot,t)$. This implies that for
$t\in [0,T_{\max})$, there holds
\begin{equation}\label{H-max}|H|^2_{\max}(t)\geq n^n\omega_n V^{-1}:=c_{10}.\end{equation}

On the other hand, by \cite{Topping}, there is a positive constant
$c_{11}$ depending only on $n$ such that for $t\in [0,T_{\max})$, we
have
\begin{equation*}
diam(M_t)\leq c_{11}\int_{M_t}|H|^{n-1}d\mu_t,
\end{equation*}
where $diam(M_t)$ denotes the diameter of $M_t$. This together with
the H\"{o}lder inequality, (\ref{11111}) and (\ref{volume}) implies
that  for $t\in [0,T_{\max})$
\begin{equation}\label{diameter}
diam(M_t)\leq
c_{11}n^{\frac{n-1}{2}}(2\Lambda)^{n-1}V^{\frac{3}{n+2}}:=c_{12}.
\end{equation}

Since $T>T_0$, we  consider the mean curvature flow on
$[\frac{T_0}{2},T_0]$.

As (\ref{improt-estimate}), we have for $t\in [\frac{T_0}{2},T_0]$
\begin{equation}\label{|phi|^2}
|\mathring{A}|\leq\bigg(1+\frac{2}{n}\bigg)^{\frac{n(n+1)(n+2)}{4p}}
c_9^{\frac{n}{2p}}\bigg(c_8p^{n+1}+\frac{(n+2)^2}{nT_0}
\bigg)^{\frac{n+2}{2p}}\cdot T_0^{\frac{1}{p}}\cdot 2\varepsilon
:=c_{13}\varepsilon.
\end{equation}
Here $c_{13}$ depends on $n,p,V,\Lambda$ and is independent of
$\varepsilon$.

 For $u=|A|^2$, since $c_1\leq c_3$, we have by (\ref{evo-f})
\begin{equation}\label{evo-u}
\frac{\partial}{\partial t}u\leq \Delta u+c_3|A|^2u.
\end{equation}
Then by a standard Moser iteration process as for $h_\sigma$ in Step
1, we have for $t\in [\frac{T_0}{2},T_0]$
\begin{equation}\label{|A|^2}
|A|^2\leq\bigg(1+\frac{2}{n}\bigg)^{\frac{n(n+1)}{2}}
c_{15}^{\frac{n}{n+2}}\bigg(c_{14}\Big(\frac{n+2}{2}\Big)^{n+1}+\frac{(n+2)^2}{nT_0}
\bigg)\cdot T_0^{\frac{2}{n+2}}\cdot 2\Lambda:=c_{16}.
\end{equation}
Here
$c_{14}=\max\{\frac{c_52^n}{(n+2)^n},\frac{c_62^{\frac{n}{2}}}{(3n)^{\frac{n}{2}}}\}$,
and $c_{15}=C_n^{\frac{n-2}{n}}\cdot
\max\{1,n^{\frac{n+2}{2}}(2\Lambda)^{n+2}T_0\cdot\frac{n+2}{n}\}.$

Set
\begin{equation*}G=\Big(t-\frac{T_0}{2}\Big)|\nabla\mathring{A}|^2+|\mathring{A}|^2.
\end{equation*}
We consider the evolution inequality of $G$ on
$[\frac{T_0}{2},T_0]$.

As in \cite{Andrews-Baker}, we have
\begin{equation*}\nabla_t(\nabla \mathring{A})=\nabla(\nabla_t\mathring{A})+A\ast
A\ast\nabla A.
\end{equation*}
Here $\nabla$ is the connection on the spatial vector bundle, which
for each $t$ is agree with the Levi-Civita connection of $g(t)$. The
evolution equation of $\mathring{A}$ is
\begin{equation*}\nabla_t\mathring{A}=\triangle\mathring{A}+A\ast
A\ast A.
\end{equation*}
On the other hand, we have
\begin{equation*}\nabla(\triangle\mathring{A})=\triangle(\nabla\mathring{A})+A\ast
A\ast\nabla A.
\end{equation*}
 Hence
\begin{equation*}\nabla_t(\nabla\mathring{A})=\triangle(\nabla\mathring{A})+A\ast
A\ast\nabla A.
\end{equation*}
This implies
\begin{equation}\label{evo- nabaphi}
\frac{\partial}{\partial t}|\nabla\mathring{A}|^2 \leq
\triangle|\nabla\mathring{A}|^2+c_{17}|A|^2|\nabla\mathring{A}|^2,
\end{equation}
where  $c_{17}$ is a positive constant depending only on $n$.  Here
we have used the inequality $|\nabla A|^2\leq \frac{3n}{2(n-1)}
|\nabla\mathring{A}|^2 $, which was proved in \cite{Andrews-Baker}.

Combining (\ref{evo-phi2}) and (\ref{evo- nabaphi}) we have

\begin{equation}\label{evo- nabaG}
\frac{\partial}{\partial t}G \leq \triangle
G+\bigg(\Big(t-\frac{T_0}{2}\Big)c_{17}|A|^2-1\bigg)|\nabla\mathring{A}|^2+c_{3}|A|^2|\mathring{A}|^2.
\end{equation}

From (\ref{|phi|^2}), (\ref{|A|^2}) and (\ref{evo- nabaG}), we have
for $t\in [\frac{T_0}{2},T_0]$

\begin{equation}\label{evo- nabaG'}
\frac{\partial}{\partial t}G \leq \triangle
G+\bigg(\Big(t-\frac{T_0}{2}\Big)c_{17}c_{16}-1\bigg)|\nabla\mathring{A}|^2+c_{3}c_{16}c_{13}^2\varepsilon^2.
\end{equation}
Set $T_3=\min\{T_0,\frac{T_0}{2}+\frac{1}{c_{17}c_{16}}\}$. Then
$\frac{T_0}{2}\leq T_3\leq T_0$. For  $t\in [\frac{T_0}{2},T_3]$, we
have from (\ref{evo- nabaG'})
\begin{equation*}
\frac{\partial}{\partial t}G \leq \triangle
G+c_{3}c_{16}c_{13}^2\varepsilon^2.
\end{equation*}
By the maximal principle, this implies
\begin{equation*}G(t)-G\Big(\frac{T_0}{2}\Big)\leq
c_{3}c_{16}c_{13}^2\Big(t-\frac{T_0}{2}\Big)\varepsilon^2\end{equation*}
for $t\in [\frac{T_0}{2},T_3]$. Hence
\begin{equation*}
\begin{split}
\Big(t-\frac{T_0}{2}\Big)|\nabla\mathring{A}|^2 \leq&
|\mathring{A}|^2\Big(\frac{T_0}{2}\Big)+
c_{3}c_{16}c_{13}^2\Big(t-\frac{T_0}{2}\Big)\varepsilon^2\\ \leq&
c_{13}^2\varepsilon^2+c_{3}c_{16}c_{13}^2\Big(t-\frac{T_0}{2}\Big)\varepsilon^2.
\end{split}
\end{equation*}
Then for $t\in (\frac{T_0}{2},T_3]$, there holds
\begin{equation}\label{aaaaaaa}
|\nabla\mathring{A}|^2 \leq
\frac{c_{13}^2}{\Big(t-\frac{T_0}{2}\Big)}\varepsilon^2+c_{3}c_{16}c_{13}^2\varepsilon^2.
\end{equation}
On the other hand, from \cite{Andrews-Baker}, we know that $|\nabla
H|^2\leq \frac{3n^2}{2(n-1)}|\nabla\mathring{A}|^2$. Therefore,
(\ref{aaaaaaa}) implies that at $t=T_3$, we have
\begin{equation}\label{bbbbbbb}
|\nabla H|^2 \leq\frac{3n^2}{2(n-1)}\cdot\Bigg(
\frac{c_{13}^2}{\Big(T_3-\frac{T_0}{2}\Big)}+c_{3}c_{16}c_{13}^2\Bigg)\varepsilon^2:=c_{18}^2\varepsilon^2.
\end{equation}
Now we consider the submanifold $M_{T_3}$ at time $T_3$. Let $x,y\in
M_{T_3}$ be two points such that
$|H|(x,T_3)=|H|_{\min}(T_3):=\min_{M_{T_3}}|H|(\cdot,T_3)$ and
$|H|(y,T_3)=|H|_{\max}(T_3):=\max_{M_{T_3}}|H|(\cdot,T_3)$. Let
$l:[0,L]\rightarrow M_{T_3}$ be the shortest geodesic such that
$l(0)=x$ and $l(L)=y$. Define a function $\eta:[0,L]\rightarrow R$
by $\eta(s)=|H|^2(l(s),T_3)$ for $s\in [0,L]$. Then
$\eta(0)=|H|^2_{\min}(T_3)$ and $\eta(L)=|H|^2_{\max}(T_3)$. By the
definition of $\eta$, we have
\begin{equation*}\bigg|\frac{d}{ds}\eta(s)\bigg|=\bigg|\frac{d}{ds}|H|^2(l(s),T_3)\bigg|\leq
\bigg|(\nabla|H|^2)(l(s),T_3)\bigg|\leq\bigg|2(|H||\nabla
H|)(l(s),T_3)\bigg|.
\end{equation*}
This together with (\ref{|A|^2}) and (\ref{bbbbbbb}) implies
\begin{equation}\bigg|\frac{d}{ds}\eta(s)\bigg|\leq
2n^{\frac{1}{2}}c_{16}c_{18}\varepsilon.
\end{equation}
Then we have
\begin{equation}\label{4444}
\eta(L)-\eta(0)=\int_{0}^{L}\frac{d}{ds}\eta ds\leq
diam(M_{T_3})\cdot 2n^{\frac{1}{2}}c_{16}c_{18}\varepsilon.
\end{equation}
Combining (\ref{H-max}), (\ref{diameter}) and (\ref{4444}), we
obtain
\begin{equation}\label{H-min}
|H|^2_{\min}(T_3)\geq c_{10}-c_{19}\varepsilon,
\end{equation}
where $c_{19}=2n^{\frac{1}{2}}c_{16}c_{18}c_{12}$. We put
\begin{equation*}\varepsilon_1=\frac{c_{10}}{2c_{19}}.\end{equation*}
Then if $\varepsilon\leq\varepsilon_1$, (\ref{H-min}) implies that
\begin{equation}\label{H-min'}
|H|^2_{\min}(T_3)\geq\frac{c_{10}}{2}.
\end{equation}

\textbf{Step 3.} In this step, we finish the proof of Theorem
\ref{convergence}.

Consider the submanifold $M_{T_3}$. Set
\begin{equation*}\varepsilon_2=\frac{c_{10}^{\frac{1}{2}}}{[2n(n-1)]^{\frac{1}{2}}c_{13}}\
\ for \ \ n\geq4,\ \ and\ \ \varepsilon_2=\frac{c_{10}^{\frac{1}{2}}}{3\sqrt{2}c_{13}}\ \
for \ \ n=3.
\end{equation*}
By (\ref{|phi|^2}) and (\ref{H-min'}), we see that if
$\varepsilon\leq \min\{\varepsilon_1,\varepsilon_2\}$, then
\begin{equation*}|A|^2(T_3)\leq c_{13}^2\varepsilon_2^2+\frac{1}{n}|H|^2(T_3)
\leq\frac{|H|^2(T_3)}{n-1}\ \ for \ \ n\geq4,\end{equation*} and
\begin{equation*}|A|^2(T_3)\leq \frac{4}{9}|H|^2(T_3)\ \ for \ \ n=3.\end{equation*}
We pick $C_{1}=\min\{\varepsilon_1,\varepsilon_2\}$, which depends
only on $n,\ p,\ V$ and $\Lambda$. Then by the uniqueness of the
mean curvature flow and the convergence theorem proved in
\cite{Andrews-Baker}, we conclude that the mean curvature flow with
initial value $F_0$ converges to a round point in finite time. This
completes the proof of Theorem \ref{convergence}.

\end{proof}

\begin{coro}\label{convergence-A'}
Let $F_0:M^n\rightarrow R^{n+d}$ $(n\geq3)$ be a smooth closed
submanifold. Suppose that the mean curvature is nowhere vanishing.
Then for any fixed $p>1$, there is a positive constant $C_1'$
depending on $n,\ p$, $\min_{M_0}|H|$ and $||A||_{n+2}$, such that
if
\begin{equation*}||\mathring{A}||_{p}<C_1',\end{equation*}
then the mean curvature flow with $F_0$ as initial value has a
unique solution $F:M\times[0,T)\rightarrow R^{n+d}$ in a finite
maximal time interval, and $F_t$ converges uniformly to a point
$x\in R^{n+d}$ as $t\rightarrow T$. The rescaled maps
$\widetilde{F}_t=\frac{F_t-x}{\sqrt{2n(T-t)}}$ converge in
$C^{\infty}$ to a limiting embedding $\widetilde{F}_T$  such that
$\widetilde{F}_T(M)$ is the unit $n$-sphere in some
$(n+1)$-dimensional subspace of $R^{n+d}$.\end{coro}

\begin{proof} It is easily to see that we can choose $C_1$ in Theorem  \ref{convergence}
 such that $C_1=C_1(n,p,V,||A||_{n+2})$  depending on
$n,\ p$,
 $||A||_{n+2}$ and the upper bound $V$ of the volume of $M_0$.  Since

\begin{equation*}||A||_{n+2}\geq n^{\frac{1}{2}} ||H||_{n+2}\geq n^{\frac{1}{2}}
 Vol(M_0)^{\frac{1}{n+2}}\min_{M_0}|H|,
\end{equation*}
 we have
  \begin{equation*}Vol(M)\leq
 n^{-\frac{n+2}{2}} (\min_{M_0}|H|)^{-(n+2)}||A||_{n+2}^{n+2}:=V'.
 \end{equation*}
Then by  Theorem  \ref{convergence}, we can pick
$C_1'=C_1(n,p,V',||A||_{n+2})$, which depends on $n,\ p$,
$\min_{M_0}|H|$ and $||A||_{n+2}$.

\end{proof}

\begin{theorem}\label{convergence-H}
Let $F_0:M^n\rightarrow R^{n+d}$ $(n\geq3)$ be a smooth closed
submanifold.  Then for any fixed $p>n$, there is a positive constant
$C_2$ depending on $n,p, Vol(M_0)$ and $||H||_{n+2}$, such that if
\begin{equation*}||\mathring{A}||_{p}<C_2,\end{equation*}then the mean curvature flow with $F_0$
as initial value has a unique solution $F:M\times[0,T)\rightarrow
R^{n+d}$ in a finite maximal time interval, and $F_t$ converges
uniformly to a point $x\in R^{n+d}$ as $t\rightarrow T$. The
rescaled maps $\widetilde{F}_t=\frac{F_t-x}{\sqrt{2n(T-t)}}$
converge in $C^{\infty}$ to a limiting embedding $\widetilde{F}_T$
such that $\widetilde{F}_T(M)$ is the unit $n$-sphere in some
$(n+1)$-dimensional subspace of $R^{n+d}$.
\end{theorem}

\begin{proof}The idea to prove Theorem \ref{convergence-H} is similar to the proof of Theorem \ref{convergence}. We set $\Lambda=||H||_{n+2}$.
Suppose
\begin{equation}
||\mathring{A}||_p<\varepsilon
\end{equation} for some fixed $p>n$ and $\varepsilon\in (0,100]$.
Set
\begin{equation*}T'=\sup\{t\in [0,T_{\max}):
 ||H||_{n+2}<\Lambda,
||\mathring{A}||_p<2\varepsilon\}.
\end{equation*}
As in the proof of Theorem \ref{convergence}, we consider the mean
curvature flow on the time interval $[0,T')$.

For $|H|^2$, we have the following inequality (see
\cite{Andrews-Baker,WaM1} for the derivation)
\begin{equation*}
\frac{\partial}{\partial t} |H|^2\leq\Delta |H|^2 -2|\nabla
H|^2+c_{20}|A|^2|H|^2,
\end{equation*} for some positive constant $c_{20}$ depending only on $n$.
Set $w=|H|^2$. Then
\begin{equation}\label{w-H}
\frac{\partial}{\partial t}w\leq \Delta
w+c_{20}|\mathring{A}|^2w+\frac{c_{20}}{n}w^2.
\end{equation}
From (\ref{w-H}) we have for $r>1$
\begin{equation}\label{w-integral111}
\begin{split}
\frac{1}{r}\frac{\partial}{\partial t}\int_{M_t} w^{r}d\mu_t
\leq&-\frac{4(r-1)}{r^2}\int_{M_t}|\nabla w
^{\frac{r}{2}}|^2d\mu_t\\
&+c_{20}\int_{M_t}|\mathring{A}|^2w^{r}d\mu_t+\frac{c_{20}}{n}\int_{M_t}w^{r+1}d\mu_t.
\end{split}
\end{equation}

Now we let $r=\frac{n+2}{2}$. As in (\ref{ineq-fp+1-1}), we have

\begin{equation}\label{ineq-w}
\begin{split}
\int_{M_t} w^{\frac{n+2}{2}+1}d\mu_t
\leq&C_n^{\frac{n}{n+2}}\bigg(\int_{M_t}w^{\frac{n+2}{2}}d\mu_t\bigg)^{2}+C_n^{\frac{n}{n+2}}\cdot
\frac{2}{n+2}\epsilon^{\frac{n+2}{2}}\bigg(\int_{M_t}w^{\frac{n+2}{2}}d\mu_t\bigg)^{2}\\
&+C_n^{\frac{n}{n+2}}\cdot\frac{n}{n+2}\epsilon^{-\frac{n+2}{n}}\int_{M_t}|\nabla
w^{\frac{n+2}{4}}|^2d\mu_t,
\end{split}
\end{equation}
for any $\epsilon>0$.

As in (\ref{second-term}), we have

\begin{equation}
\begin{split}\label{second-term'}
\int_{M_t}|\mathring{A}|^2w^{\frac{n+2}{2}}d\mu_t
\leq&(200)^2C_{n}^{\frac{n}{n+2}}\int_M
w^{\frac{n+2}{2}}d\mu_t\\
&+(200)^{n+2}C_{n}^{\frac{n}{n+2}}\cdot\frac{2}{n+2}\mu^{\frac{n+2}{2}}\int_{M_t}w^{\frac{n+2}{2}}
d\mu_t\\
&+(200)^{n+2}C_{n}^{\frac{n}{n+2}}\cdot\frac{n}{n+2}\mu^{-\frac{n+2}{n}}\int_M|\nabla
w^{\frac{n+2}{4}}|^2d\mu_t,
\end{split}
\end{equation}
for any $\mu>0$.

Therefore, combining (\ref{w-integral111}), (\ref{ineq-w}) and
(\ref{second-term'}) we have

\begin{equation}\label{w-integral}
\begin{split}
\frac{2}{n+2}\cdot\frac{\partial}{\partial t}\int_{M_t} w^{\frac{n+2}{2}}d\mu_t
\leq&\bigg(c_{20}(200)^{n+2}C_{n}^{\frac{n}{n+2}}\cdot\frac{n}{n+2}\mu^{-\frac{n+2}{n}}
+\frac{c_{20}}{n}\cdot C_n^{\frac{n}{n+2}}\cdot\frac{n}{n+2}\epsilon^{-\frac{n+2}{n}}\\
&-\frac{8n}{(n+2)^2}\bigg)\int_{M_t}|\nabla w
^{\frac{r}{2}}|^2d\mu_t\\
&+ c_{20}\bigg((200)^2C_{n}^{\frac{n}{n+2}}+(200)^{n+2}C_{n}^{\frac{n}{n+2}}\cdot\frac{2}{n+2}\mu^{\frac{n+2}{2}}\bigg)\int_{M_t}w^{\frac{n+2}{2}}
d\mu_t\\
&+\frac{c_{20}}{n}\bigg(C_n^{\frac{n}{n+2}}+C_n^{\frac{n}{n+2}}\cdot
\frac{2}{n+2}\epsilon^{\frac{n+2}{2}}\bigg)\bigg(\int_{M_t}w^{\frac{n+2}{2}}d\mu_t\bigg)^{2}.
\end{split}
\end{equation}
Now we pick
\begin{equation*}\mu=\epsilon=\bigg(\frac{c_{20}(200)^{n+2}C_{n}^{\frac{n}{n+2}}\cdot\frac{n}{n+2}
+\frac{c_{20}}{n}\cdot C_n^{\frac{n}{n+2}}\cdot\frac{n}{n+2}}{\frac{6n-4}{(n+2)^2}}\bigg)^{\frac{n}{n+2}}.\end{equation*}
Then from (\ref{w-integral}), we have
\begin{equation}\label{w-integral'}
\frac{\partial}{\partial t}\int_{M_t} w^{\frac{n+2}{2}}d\mu_t
\leq c_{21}\int_{M_t}w^{\frac{n+2}{2}}
d\mu_t+c_{22}\bigg(\int_{M_t}w^{\frac{n+2}{2}}d\mu_t\bigg)^{2},
\end{equation}
where $c_{21}$ and   $c_{22}$ are positive constants depending only on $n$.

Let $\rho(t)$ be the positive solution to the following Bernoulli equation

\begin{equation*}
\begin{split}\frac{d}{dt}\rho=&c_{21}\rho+c_{22}\rho^2,\\
\rho(0)=&\Lambda^{n+2}.
\end{split}
\end{equation*}
Then
\begin{equation*}
\rho(t)=\frac{e^{c_{21}t}}{\frac{1}{\Lambda^{n+2}}+\frac{c_{22}}{c_{21}}-\frac{c_{22}}{c_{21}}e^t},
 \ \ t\in \bigg[0,\frac{2\ln\Big(\frac{c_{21}}{c_{22}\Lambda^{n+2}}+1\Big)}{2c_{21}}\bigg).
\end{equation*}
Let $T_1'>0$ such that $\rho(t)\leq (2\Lambda)^{n+2}$ for $t\in
[0,T_1']$. Then by the maximal principle, we see that for $t\in
[0,\min\{T',T_1'\})$, there holds
\begin{equation*}
\int_{M_t} w^{\frac{n+2}{2}}d\mu_t<
\Big(\frac{3}{2}\Lambda\Big)^{n+2},
\end{equation*}
or equivalently,
\begin{equation}\label{||H(t)||_{n+2}'}
||H(t)||_{n+2}<\frac{3}{2}\Lambda.
\end{equation}

Next, from (\ref{evo-h}) we have

\begin{equation}\label{evo-h-H-phi}
\frac{\partial}{\partial t}h_\sigma\leq\Delta h_\sigma
+c_3|\mathring{A}|^2h_\sigma+\frac{c_3}{n}|H|^2h_\sigma.
\end{equation}

From (\ref{evo-h-H-phi}) we have for $r>1$
\begin{equation}\label{f-integral-H}
\begin{split}
\frac{1}{r}\frac{\partial}{\partial t}\int_{M_t} h_\sigma^{r}d\mu_t
\leq&-\frac{4(r-1)}{r^2}\int_{M_t}|\nabla h_\sigma
^{\frac{r}{2}}|^2d\mu_t\\
&+c_3\int_{M_t}|\mathring{A}|^2h_\sigma^{r}d\mu_t+\frac{c_3}{n}\int_{M_t}|H|^2h_\sigma^{r}d\mu_t.
\end{split}
\end{equation}

As in (\ref{second-term}), we have for $r\geq p>n$, there holds

\begin{equation}
\begin{split}\label{phi.h_sigma^p}
\int_{M_t}|\mathring{A}|^2h_\sigma^{r}d\mu_t
\leq&(200)^2C_{n}^{\frac{n}{p}}\int_M
h_\sigma^rd\mu_t\\
&+(200)^{p}C_{n}^{\frac{n}{p}}\cdot\frac{p-n}{p}\nu^{\frac{p}{p-n}}\int_{M_t}h_\sigma^r
d\mu_t\\
&+(200)^{p}C_{n}^{\frac{n}{p}}\cdot\frac{n}{p}\nu^{-\frac{p}{n}}\int_M|\nabla
h_\sigma^{\frac{r}{2}}|^2d\mu_t,
\end{split}
\end{equation}
 and
 \begin{equation}
\begin{split}\label{phi.h-H_sigma^p}
\int_{M_t}|H|^2h_\sigma^{r}d\mu_t
\leq&(200)^2C_{n}^{\frac{n}{n+2}}\int_M
h_\sigma^rd\mu_t\\
&+(200)^{n+2}C_{n}^{\frac{n}{n+2}}\cdot\frac{2}{n+2}\varrho^{\frac{n+2}{2}}\int_{M_t}h_\sigma^r
d\mu_t\\
&+(200)^{n+2}C_{n}^{\frac{n}{n+2}}\cdot\frac{n}{n+2}\varrho^{-\frac{n+2}{n}}\int_M|\nabla
h_\sigma^{\frac{r}{2}}|^2d\mu_t,
\end{split}
\end{equation}
for any $\nu,\ \varrho>0$.

From (\ref{f-integral-H}), (\ref{phi.h_sigma^p}) and
(\ref{phi.h-H_sigma^p}), we have
\begin{equation}\label{f-integral-H2}
\begin{split}
\frac{1}{r}\frac{\partial}{\partial t}\int_{M_t} h_\sigma^{r}d\mu_t
\leq&\bigg(c_3(200)^{p}C_{n}^{\frac{n}{p}}\cdot\frac{n}{p}\nu^{-\frac{p}{n}}
+\frac{c_3}{n}(200)^{n+2}C_{n}^{\frac{n}{n+2}}\cdot\frac{n}{n+2}\varrho^{-\frac{n+2}{n}}\\
&-\frac{4(r-1)}{r^2}\bigg)\int_{M_t}|\nabla
h_\sigma^{\frac{r}{2}}|^2d\mu_t\\
&+\bigg(c_{3}(200)^2C_{n}^{\frac{n}{p}}+c_{3}(200)^{p}C_{n}^{\frac{n}{p}}\cdot\frac{p-n}{p}\nu^{\frac{p}{p-n}}\\
&+\frac{c_3}{n}\cdot (200)^2C_{n}^{\frac{n}{n+2}}+\frac{c_3}{n}\cdot (200)^{n+2}C_{n}^{\frac{n}{n+2}}\cdot\frac{2}{n+2}\varrho^{\frac{n+2}{2}}\bigg)
\int_{M_t}h_\sigma^r
d\mu_t.
\end{split}
\end{equation}

Pick
\begin{equation*}\nu^{\frac{p}{n+2}}=\varrho=\bigg(\frac{c_3(200)^{p}C_{n}^{\frac{n}{p}}\cdot\frac{n}{p}
+\frac{c_3}{n}(200)^{n+2}C_{n}^{\frac{n}{n+2}}\cdot\frac{n}{n+2}}{\frac{3r-4}{r^2}}\bigg)^{\frac{n}{n+2}}.\end{equation*}
Since $r\geq p>n$, then
\begin{equation*}\nu^{\frac{p}{n+2}}=\varrho\leq\bigg(\frac{c_3(200)^{p}C_{n}^{\frac{n}{p}}\cdot\frac{n}{p}
+\frac{c_3}{n}(200)^{n+2}C_{n}^{\frac{n}{n+2}}\cdot\frac{n}{n+2}}{3p-4}\bigg)^{\frac{n}{n+2}}\cdot
r^{\frac{2n}{n+2}}:=c_{23}\cdot r^{\frac{2n}{n+2}}.\end{equation*}
 Then from
(\ref{f-integral-H2}), we have

\begin{equation}\label{f-integral-H2'}
\begin{split}
\frac{\partial}{\partial t}\int_{M_t} h_\sigma^{r}d\mu_t
+\int_{M_t}|\nabla h_\sigma^{\frac{r}{2}}|^2d\mu_t \leq
c_{24}r^{\frac{p+n}{p-n}}\int_{M_t}h_\sigma^r d\mu_t,
\end{split}
\end{equation}
where
\begin{equation*}
\begin{split}c_{24}=\max\Big\{&c_{3}(200)^2C_{n}^{\frac{n}{p}}+
\frac{c_3}{n}\cdot (200)^2C_{n}^{\frac{n}{n+2}},\ \
c_{23}^{\frac{n+2}{p-n}}\cdot
c_{3}(200)^{p}C_{n}^{\frac{n}{p}}\cdot\frac{p-n}{p},\\
&c_{23}^{\frac{n+2}{2}}\cdot  \frac{c_3}{n}\cdot
(200)^{n+2}C_{n}^{\frac{n}{n+2}}\cdot\frac{2}{n+2} \Big\}.
\end{split}\end{equation*}

Letting $r=p$, we have from (\ref{f-integral-H2'})

\begin{equation}\label{f-integral-H2''}
\begin{split}
\frac{\partial}{\partial t}\int_{M_t} h_\sigma^{p}d\mu_t
 \leq
c_{24}p^{\frac{p+n}{p-n}}\int_{M_t}h_\sigma^p d\mu_t,
\end{split}
\end{equation}
Now we apply  the maximal principle and let $\sigma\rightarrow 0$.
Then for $t\in [0,\min \{T', T_2'\})$, where
$T_2'=c_{24}^{-1}p^{-\frac{2n}{p-n}-1}$, there holds
\begin{equation*}||\mathring{A}(t)||_p< \frac{3}{2}\varepsilon.\end{equation*}

Set $T_0'=\min\{T_1',T_2'\}$. As in the Step 1 of the proof  of
Theorem \ref{convergence}, we can prove that $T'>T_0'$ by contradiction. In fact,
from the smoothness of the mean curvature flow we exclude the case
where $T'<T_{\max}$. For the case where  $T'=T_{\max}$, since we have
(\ref{f-integral-H2'}), which has similar form as (\ref{ineq-h2}),
we can apply the standard Moser process to obtain the following
estimate for small $\theta>0$.

\begin{equation}\label{improt-estimate'}
h_\sigma(x,t)\leq\bigg(1+\frac{2}{n}\bigg)^{\frac{n(n+2)(p+n)}{4p(p-n)}}
c_{25}^{\frac{n}{2p}}\bigg(c_{24}p^{\frac{p+n}{p-n}}+\frac{(n+2)^2}{2nt}
\bigg)^{\frac{n+2}{2p}}\bigg(\int_{0}^{T_{\max}-\theta}\int_{M_t}h_\sigma^{p}d\mu_tdt\bigg)^{\frac{1}{p}}.
\end{equation}
Here $c_{25}=C_n^{\frac{n-2}{n}}\cdot
\max\{1,(2\Lambda)^{n+2}T_0'\}$.

Now we let $\sigma\rightarrow 0$ and  $\theta\rightarrow 0$. Then we
have for $t\in [\frac{T_{\max}}{2},T_{\max})$,
\begin{equation*}
|\mathring{A}|^2(x,t)\leq C'(n,p,\Lambda,\varepsilon,T_{\max})<+\infty.
\end{equation*}
This implies that
\begin{equation*}|A|^2\leq a'|H|^2+b'
\end{equation*}
on $[0,T_{\max})$ for some positive constants $a'$ and $b'$
independent of $t$. On the other hand we also have
$\int_{0}^{T_{\max}}\int_{M_t}|H|^{n+2}d\mu_tdt<+\infty.$ Applying
Theorem \ref{extension} we  conclude that the mean curvature flow
can be extended over time $T_{\max}$. This is a contradiction.

We consider the mean curvature flow for $t\in
[\frac{T_0'}{2},T_0']$. As (\ref{improt-estimate'}), we have
\begin{equation}\label{improt-estimate''}
|\mathring{A}|(x,t)\leq\bigg(1+\frac{2}{n}\bigg)^{\frac{n(n+2)(p+n)}{4p(p-n)}}
c_{25}^{\frac{n}{2p}}\bigg(c_{24}p^{\frac{p+n}{p-n}}+\frac{(n+2)^2}{nT_0'}
\bigg)^{\frac{n+2}{2p}}T_0'^{\frac{1}{p}}\cdot2\varepsilon:=c_{27}\varepsilon.
\end{equation}

By (\ref{w-H}), we have
\begin{equation*}
\frac{\partial}{\partial t}w\leq \Delta
w+c_{20}|\mathring{A}|^2w+\frac{c_{20}}{n}|H|^2w.
\end{equation*}
Then similarly as (\ref{improt-estimate''}), we get for $t\in
[\frac{T_0'}{2},T_0']$
\begin{equation}\label{improt-estimate'''}
|H|^2(x,t)\leq\bigg(1+\frac{2}{n}\bigg)^{\frac{n(n+1)}{2}}
c_{28}^{\frac{n}{n+2}}\bigg(c_{27}(n+2)^{n+1}+\frac{(n+2)^2}{nT_0'}
\bigg)T_0'^{\frac{2}{n+2}}\cdot(2\Lambda)^2:=c_{29}.
\end{equation}
Here
\begin{equation*}
\begin{split}c_{27}=\max\Big\{&c_{20}(200)^2C_{n}^{\frac{n}{n+2}}+
\frac{c_{20}}{n}\cdot (200)^2C_{n}^{\frac{n}{n+2}},\ \
c_{23}'^{\frac{n+2}{2}}\cdot
c_{20}(200)^{n+2}C_{n}^{\frac{n}{n+2}}\cdot\frac{2}{n+2},\\
&c_{23}'^{\frac{n+2}{2}}\cdot  \frac{c_{20}}{n}\cdot
(200)^{n+2}C_{n}^{\frac{n}{n+2}}\cdot\frac{2}{n+2} \Big\}.
\end{split}\end{equation*}
\begin{equation*}c_{28}=C_n^{\frac{n-2}{n}}\cdot \max\{1,(2\Lambda)^{n+2}T_0'\},\end{equation*}
and
\begin{equation*}c_{23}'=\bigg(\frac{c_3(200)^{n+2}C_{n}^{\frac{n}{n+2}}\cdot\frac{n}{n+2}
+\frac{c_3}{n}(200)^{n+2}C_{n}^{\frac{n}{n+2}}\cdot\frac{n}{n+2}}{3n+2}\bigg)^{\frac{n}{n+2}}.\end{equation*}

By (\ref{improt-estimate''}) and (\ref{improt-estimate'''}), we have
\begin{equation}\label{improt-estimate''''}
|A|^2(x,t)\leq c_{27}^2100^2+\frac{c_{29}}{n}:=c_{30},
\end{equation}
for $t\in [\frac{T_0'}{2},T_0']$. As in Step 2 of the proof of
Theorem \ref{convergence}, we have for $t\in [0,T_{\max})$, there
hold
\begin{equation}\label{H-max'}|H|^2_{\max}(t)\geq n^n\omega_n V^{-1}:=c_{31},\end{equation}
and
\begin{equation}\label{diameter'}
diam(M_t)\leq c_{11}(2\Lambda)^{n-1}V^{\frac{3}{n+2}}:=c_{32},
\end{equation}
where $V=Vol(M_0)$.

 Using a similar argument,  for $t\in
[\frac{T_0'}{2},T_3']$, where
$T_3'=\min\{T_0',\frac{T_0'}{2}+\frac{1}{c_{17}c_{30}}\}$,  we have
\begin{equation}\label{bbbbbbb'}
|\nabla H|^2 \leq\frac{3n^2}{2(n-1)}\cdot\Bigg(
\frac{c_{27}^2}{\Big(t-\frac{T_0'}{2}\Big)}+c_{3}c_{31}c_{27}^2\Bigg)\varepsilon^2:=c_{33}^2\varepsilon^2.
\end{equation}
Combining (\ref{H-max'}), (\ref{diameter'}) and (\ref{bbbbbbb'}), we
obtain that, at time $T_3'$, there is
\begin{equation*}\varepsilon_1'=\frac{c_{31}}{2n^{\frac{1}{2}}c_{30}c_{32}c_{33}},\end{equation*}
such that if $\varepsilon\leq\varepsilon_1'$, then

\begin{equation}\label{H-min''}
|H|^2_{\min}(T_3')\geq\frac{c_{31}}{2}.
\end{equation}

 Set
\begin{equation*}\varepsilon_2'=\frac{c_{31}^{\frac{1}{2}}}{[2n(n-1)]^{\frac{1}{2}}c_{27}}\
\ for \ \ n\geq4,\ \ and\ \ \varepsilon_2'=\frac{c_{31}^{\frac{1}{2}}}{3\sqrt{2}c_{27}}\ \
for \ \ n=3.
\end{equation*}
By (\ref{|phi|^2}) and (\ref{H-min'}), we see that if
$\varepsilon\leq \min\{\varepsilon_1',\varepsilon_2',100\}$, then
\begin{equation*}|A|^2(T_3')\leq c_{27}^2\varepsilon_2^2+\frac{1}{n}|H|^2(T_3')
\leq\frac{|H|^2(T_3')}{n-1}\ \ for \ \ n\geq4,\end{equation*} and
\begin{equation*}|A|^2(T_3')\leq \frac{4}{9}|H|^2(T_3')\ \ for \ \ n=3.\end{equation*}
Then we can  pick $C_{2}=\min\{\varepsilon_1',\varepsilon_2',100\}$,
which depends only on $n,\ p,\ V$ and $\Lambda$, and this completes
the proof of Theorem \ref{convergence-H}.
\end{proof}

Using a similar argument as in the proof of Corollary
\ref{convergence-A'}, we have following

\begin{coro}\label{convergence-H'}
Let $F_0:M^n\rightarrow R^{n+d}$ $(n\geq3)$ be a smooth closed
submanifold. Suppose that the mean curvature is nowhere vanishing.
Then for any fixed $p>n$, there is a positive constant $C_2'$
depending on $n,\ p$, $\min_{M_0}|H|$ and $||H||_{n+2}$, such that
if
\begin{equation*}||\mathring{A}||_{p}<C_2',\end{equation*}
then the mean curvature flow with $F_0$ as initial value has a
unique solution $F:M\times[0,T)\rightarrow R^{n+d}$ in a finite
maximal time interval, and $F_t$ converges uniformly to a point
$x\in R^{n+d}$ as $t\rightarrow T$. The rescaled maps
$\widetilde{F}_t=\frac{F_t-x}{\sqrt{2n(T-t)}}$ converge in
$C^{\infty}$ to a limiting embedding $\widetilde{F}_T$  such that
$\widetilde{F}_T(M)$ is the unit $n$-sphere in some
$(n+1)$-dimensional subspace of $R^{n+d}$.\end{coro}

\section{Open Problems}

In this section, we propose several open problems for the
convergence of the mean curvature flow of submanifolds. Denote by
$F^{n+d}(c)$ the $(n+d)$-dimensional complete simply connected space
form
 of constant sectional curvature $c$. Let $M$ be an
$n$-dimensional closed oriented submanifold in $F^{n+d}(c)$ with
$c\geq0$. Shiohama-Xu \cite{Shiohama-Xu-97} showed that if
$|A|^2<\alpha(n,H,c)$, then $M$
 is homeomorphic to a sphere for $n\geq4$, or diffeomorphic to a
 spherical space form for $n=3$. Here
 $$\alpha(n,H,c)=nc+\frac{nH^2}{2(n-1)}-\frac{n-2}{2(n-1)}\sqrt{H^2+4(n-1)cH^2}.$$
In \cite{Xu-Zhao}, Xu-Zhao proved several differentiable sphere
theorems for submanifolds satisfying suitable pinching conditions in
a Riemannian manifold. Recently, Xu-Gu \cite{XJ2} strengthened
Shiohama-Xu's topological sphere theorem for $c=0$ to be a
differentiable sphere theorem.
 Motivated by these sphere theorems and
the convergence theorem for the mean curvature flow  due to Andrews
and Baker \cite{Andrews-Baker}, we propose the following
\begin{problem}Let $M$ be an $n$-dimensional $(n\geq2)$ smooth closed
submanifold in  $F^{n+d}(c)$ with $c>0$. Let $M_t$ be the solution
of the mean curvature flow with $M$ as initial submanifold. Suppose
$M$ satisfies
\begin{eqnarray*}
|A|^2<\alpha(n,H,c).
\end{eqnarray*}
Then one of the following holds.

a) The mean curvature flow has a smooth solution $M_t$ on a finite
time interval $0\leq t<T$ and the $M_t$'s converge uniformly to a
round point as $t\rightarrow T$.

b) The mean curvature flow has a smooth solution $M_t$ for all $0
\leq t < \infty$ and the $M_t$'s converge in the $C^\infty$-topology
to a smooth totally geodesic submanifold $M_\infty$ in $F^{n+d}(c)$.

In particular, $M$ is diffeomorphic to the standard $n$-sphere.
\end{problem}

In \cite{Shiohama-Xu-94}, Shiohama-Xu obtained a topological sphere
theorem for closed submanifolds satisfying $||\mathring{A}||_{n}<C(n)$ in
$F^{n+d}(c)$ with $c\geq0$
 for an explicit positive constant
$C(n)$ depending only on $n$. The following problems arise out of
this topological sphere theorem and our convergence theorems.

\begin{problem}\label{problem}Let $M$ be an $n$-dimensional $(n\geq2)$ smooth closed
submanifold in  $R^{n+d}$ . Let $M_t$ be the solution of the mean
curvature flow with $M$ as initial submanifold. Then there exists an
positive constant $D(n)$ depending only on $n$, such that if
 $M$ satisfies
\begin{eqnarray*}
||\mathring{A}||_{n}<D(n),
\end{eqnarray*}
then the mean curvature flow has a solution $M_t$'s on a finite time
interval $[0,T)$ and $M_t$ converges uniformly to a round point. In
particular, $M$ is diffeomorphic to the standard $n$-sphere.
\end{problem}

For any $4$-dimensional compact manifold $M$ which is
homeomorphic to a sphere, we hope to show that there exists an
isometric embedding of the 4-sphere into an Euclidean space such
that $||\mathring{A}||_{4}$ is small enough in the sense of Theorems 1.2 or
Open problem \ref{problem}. In fact, Shiohama and the second author
\cite{Shiohama-Xu-94} proved that for any $4$-dimensional compact
submanifold $M$ in an Euclidean space, we have $||\mathring{A}||_{4}\ge
C(\Sigma_{i=1}^3\beta_i)^{1/4}$, where C is a universal positive
constant and $\beta_i$ is the $i$-th Betti number of $M$, $i=1, 2,
3$. Therefore it's possible to isometrically embed a topological
4-sphere into an Euclidean space with small upper bound for
$||\mathring{A}||_{4}$. If this can be done, then we can deduce that $M$ is
diffeomorphic to a sphere.  This may open a way to prove the smooth Poincar\'e conjecture
in dimension $4$ which is now one of the most challenging problems in geometry and topology.

In general, for a homotopy sphere $M$, we can try to find its embedding
in Euclidean spaces with small  integral norm $||\mathring{A}||_{n}$. Our results on mean
curvature flow of arbitrary codimension reduce the
problem of proving whether $M$ is diffeomorphic to a sphere to the problem of finding
the optimal embeddings of $M$ into Euclidean spaces.

\begin{problem}Let $M$ be an $n$-dimensional $(n\geq2)$ smooth closed
submanifold in $F^{n+d}(c)$ with $c>0$. Let $M_t$ be the solution of
the mean curvature flow with $M$ as initial submanifold. Then there
exists an positive constant $E(n)$ depending only on $n$, such that
if
 $M$ satisfies
\begin{eqnarray*}
||\mathring{A}||_{n}<E(n),
\end{eqnarray*}
then one of the following holds.

a) The mean curvature flow has a smooth solution $M_t$ on a finite
time interval $0\leq t<T$ and the $M_t$'s converge uniformly to a
round point as $t\rightarrow T$.

b) The mean curvature flow has a smooth solution $M_t$ for all $0
\leq t < \infty$ and the $M_t$'s converge in the $C^\infty$-topology
to a smooth totally geodesic submanifold $M_\infty$ in $F^{n+d}(c)$.

In particular, $M$ is diffeomorphic to the standard $n$-sphere.
\end{problem}

\end{document}